\documentclass[12pt,reqno]{amsart}

\usepackage{amssymb}
\usepackage{booktabs}
\usepackage[bookmarks=false,unicode,colorlinks,urlcolor=red,citecolor=red,linkcolor=blue]{hyperref}
\usepackage{aliascnt}
\usepackage{tikz}
\usepackage{pgfplots} 
\usepgfplotslibrary{polar}

\usepackage{color}

\usepackage[margin=1.25in]{geometry}

\newtheorem*{remark}{Remark}

\newtheorem{theorem}{Theorem}[section]

\newaliascnt{lemma}{theorem}
\newtheorem{lemma}[lemma]{Lemma}
\aliascntresetthe{lemma}

\newaliascnt{proposition}{theorem}

\aliascntresetthe{proposition}

\newaliascnt{corollary}{theorem}
\newtheorem{corollary}[corollary]{Corollary}
\aliascntresetthe{corollary}

\newaliascnt{conjecture}{theorem}

\aliascntresetthe{conjecture}

\newaliascnt{example}{theorem}
\newtheorem{example}[example]{Example}
\aliascntresetthe{example}

\makeatletter
\def\tagform@#1{\maketag@@@{\ignorespaces#1\unskip\@@italiccorr}}
\let\orgtheequation\theequation
\def\theequation{(\orgtheequation)}
\makeatother
\def\equationautorefname~{}

\newcommand{\D}{{\mathbb D}}
\newcommand{\e}{\varepsilon}
\newcommand{\Imag}{\operatorname{Im}}
\newcommand{\N}{{\mathbb N}}
\newcommand{\Ray}{\operatorname{Ray}}
\newcommand{\R}{{\mathbb R}}
\newcommand{\Real}{\operatorname{Re}}
\newcommand{\Sp}{{\mathbb S}}

\begin{document}

\title[Steklov eigenvalues and quasiconformal maps]{Steklov eigenvalues and quasiconformal maps of simply connected planar domains}
\author[]{A. Girouard, R. S. Laugesen and B. A. Siudeja}
\address{Department de Mathematiques et Statistique, Univ.\ Laval, Quebec, Qc, Canada}
\email{Alexandre.Girouard@ulaval.ca}
\address{Department of Mathematics, Univ.\ of Illinois, Urbana,
IL 61801, U.S.A.}
\email{Laugesen\@@illinois.edu}
\address{Department of Mathematics, Univ.\ of Oregon, Eugene,
OR 97403, U.S.A.}
\email{Siudeja\@@uoregon.edu}
\date{\today}


\begin{abstract}
We investigate isoperimetric upper bounds for sums of consecutive Steklov eigenvalues of planar domains. The normalization involves the perimeter and scale-invariant geometric factors which measure deviation of the domain from roundness. We prove sharp upper bounds for both starlike and simply connected domains, for a large collection of spectral functionals including partial sums of the zeta function and heat trace. The proofs rely on a special class of quasiconformal mappings.
\end{abstract}

\maketitle

\section{\bf Introduction and Results}
\label{sec:intro}

Steklov eigenvalues of planar domains describe the frequencies of vibration of a membrane with mass concentrated at the boundary. Mathematically, we let $\Omega\subset\R^2$ be a bounded planar domain with Lipschitz boundary $\Sigma=\partial \Omega$. The Steklov eigenvalue problem is to determine the real numbers $\sigma$ for which a nonzero harmonic function exists having normal derivative equal to $\sigma$ times the value on the boundary:
\begin{equation*}
\begin{cases}
  \Delta u=0& \text{ in } \Omega,\\
  \frac{\partial u}{\partial n}=\sigma q u& \text{ on
  }\Sigma,
\end{cases}
\end{equation*}
where $q\in L^\infty(\Sigma)$ is a positive weight function. The spectrum is discrete \cite{A04} and is given by a sequence of eigenvalues 
\[
0=\sigma_0 < \sigma_1 \le \sigma_2 \le \dots \nearrow \infty 
\]
that grows asymptotically like $\sigma_j \sim j \pi/\int_\Sigma q \, ds$ if $\Sigma$ and $q$ are smooth. The corresponding eigenfunctions form an orthonormal basis of $L^2(\Sigma)$. For these basic properties of the Steklov spectrum, see the recent survey \cite{GP14} and the references therein. When we want to emphasize the dependence of the eigenvalue on the domain and the weight, we will write $\sigma_j(\Omega,q)$. In the unweighted case ($q \equiv 1$), we write simply $\sigma_j(\Omega)$. 

The Steklov spectrum can rarely be computed explicitly. Even for the square the spectrum was completely determined only recently \cite{GP14}. This lack of examples makes it especially interesting to obtain good estimates on Steklov eigenvalues, as we will do in this paper. We study sums of consecutive Steklov eigenvalues, asking:
\begin{center}
  \emph{how large can the sum $\sigma_1+\cdots+\sigma_j$ be?}
\end{center}
Eigenvalue sum inequalities generate zeta function and heat trace inequalities via majorization --- see \autoref{Corollary:Special} below and its proof. In general, the sum of the first $j$ eigenvalues represents the energy for filling the lowest $j$ quantum states when at most one particle can occupy each state (the Pauli exclusion principle). Motivated by this physical interpretation, and in an attempt to prove a summed version of the P\'{o}lya conjecture, the eigenvalue sums of the Laplacian have been studied extensively through Berezin--Li--Yau inequalities \cite{HH08,LY83}, giving results that are asymptotically sharp as $j \to \infty$. In a different direction, geometrically sharp inequalities for Laplace eigenvalue sums (with fixed index $j$) were developed on starlike domains by the second and third authors \cite{LS13, LS14}. The biLaplacian was treated too \cite{S14}. 

We discover significantly stronger results for the Steklov case. Indeed, we will handle not just starlike domains but the more general class of simply connected domains. The key new idea in the paper is the introduction of quasiconformal mappings to obtain sharp eigenvalue estimates. Specifically, we transplant trial functions from the disk to a simply connected domain through a quasiconformal mapping whose complex dilatation depends only on the angular variable. In the past, conformal mappings were used for this purpose: by P\'olya--Schiffer \cite{PSc53}, P\'olya--Szeg\"o \cite{PS51} and Laugesen--Morpurgo \cite{LM98} for the Laplacian; and Dittmar \cite{D88,D04,D07}, Hersch--Payne \cite{HP68}, Weinstock \cite{W54} for the Steklov problem. Quasiconformal maps give considerably more flexibility. 

Further, the ``angular uniformization'' step in our method enables us to work with sums of eigenvalues rather than sums of reciprocals as earlier authors did; this improvement yields heat trace inequalities and more; see \autoref{Corollary:SimplyConnectedEstimate}. And we obtain smaller (hence better) constants in the Steklov situation than the original Laplacian case would predict, due to our use of an optimal stretch of the disk: the map $r \mapsto r^t$ in \autoref{sec:underlying}. Consequently one reduces from an arithmetic mean of two constants to a geometric mean, for example from $(1+\gamma^2)/2$ to $\gamma$ in \autoref{Corollary:SimplyConnectedEstimate} below.

Historically, Steklov introduced the eigenvalue problem in 1902~\cite{St}. It can be interpreted also in terms of sloshing of a liquid \cite{Ketal,P67}. In the unweighted case, the Steklov spectrum coincides with that of the Dirichlet-to-Neumann operator 
$f \mapsto \partial_n(\mathcal H f)$, where $\mathcal Hf$ is the unique harmonic extension of $f$ from $\partial \Omega$ into the interior of $\Omega$. This Dirichlet-to-Neumann operator arises in numerous inverse problems \cite{SU90}. 

\subsection*{Spectrum of the disk}
The unweighted Steklov spectrum ($q \equiv 1$) of the unit disk $\D$ is well known to be $0,1,1,2,2,3,3,\ldots$. That is, 
\begin{equation} \label{eq:diskspectrum}
\sigma_j(\D) = \left\lceil \frac{j}{2} \right\rceil , \qquad j \geq 0 .
\end{equation}
Each positive eigenvalue $\sigma=k$ has multiplicity $2$, with eigenfunctions 
\begin{equation} \label{eq:diskeigenfns}
u = r^k\cos(k\theta), \qquad u = r^k\sin(k\theta),
\end{equation}
that are harmonic on the disk and satisfy $\frac{\partial u}{\partial r}=ku$ on the unit circle. 

\subsection*{Quasiconformal mappings of the disk, and the main result}

Recall the Wirtinger derivatives
\[
\partial f = \frac{1}{2} (f_x - i f_y) , \qquad \overline{\partial} f = \frac{1}{2} (f_x + i f_y) .
\]
A homeomorphism $f$ of the unit disk $\D$ onto a planar domain $\Omega$ is \emph{quasiconformal} if $f$ is absolutely continuous on lines and 
\[
\overline{\partial} f = \mu \, \partial f \qquad \text{a.e.\ in $\D$}
\]
for some $\mu \in L^\infty(\D)$ with $\lVert \mu \rVert_{L^\infty(\D)} < 1$. Recall that $\overline{\partial} f = \mu \, \partial f$ is known as the \emph{Beltrami equation}, and $\mu$ is called the \emph{complex dilatation}. For more information on quasiconformal mappings, see the book of Lehto and Virtanen \cite[Chapter IV]{LV73}. 

A simplifying assumption in this paper is that: 
\[
\text{the complex dilatation $\mu$ depends only on the angular variable $\theta$.}
\]
This assumption fails in general, but it does hold for conformal mappings, where $\mu \equiv 0$, and for certain starlike mappings (see \autoref{ex:starlike}). Under this angular assumption we define 
\begin{equation} \label{eq:a0a1are}
a_0(\theta) = \frac{|e^{2i\theta}-\mu(e^{i\theta})|^2}{1-|\mu(e^{i\theta})|^2} , \qquad
a_1(\theta) = \frac{|e^{2i\theta}+\mu(e^{i\theta})|^2}{1-|\mu(e^{i\theta})|^2} .
\end{equation}
Then let
\begin{equation} 
g_0 = \frac{1}{2\pi} \int_0^{2\pi} a_0(\theta) \, d\theta , \qquad 
g_1 = \frac{\frac{1}{2\pi}\int_0^{2\pi} a_1(\theta) p(\theta)^2 \,d\theta}
  {\big(\frac{1}{2\pi}\int_0^{2\pi} p(\theta) \,d\theta\big)^{\! 2}} , \label{eq:g01is}
\end{equation}
where the weight function
\[
p(\theta)=q(f(e^{i\theta})) |\partial_\theta f(e^{i \theta})|
\]
on the unit circle has been defined by requiring it to push forward under $f$ to the weight $q$ on $\Sigma$. (We assume $f : \partial \D \to \partial \Omega$ is absolutely continuous, so that the last formula makes sense a.e.) 
Assuming $p \in L^2[0,2\pi]$, we have $g_1 < \infty$. Clearly $p$ has total mass 
\[
\int_0^{2\pi} p \, d\theta = \int_\Sigma q \, ds = L(\Sigma,q),
\]
which is the $q$-weighted length of the boundary $\Sigma$. 
\begin{lemma} \label{le:ggeq1}
Under the assumptions above, one has $g_0 g_1 \geq 1$. 

Equality statement: assuming the Beltrami equation holds also on the unit circle, one has that $g_0 g_1 = 1$ if and only if $e^{-2i\theta} \mu \in (-1,1)$ and $|\partial_r f| (q \circ f) = \text{constant}$ almost everywhere on the unit circle. 
\end{lemma}
The lemma is proved in \autoref{sec:underlying}.

For example, if $f$ is conformal on the closed disk then $\mu \equiv 0$ and $|\partial_r f| = |f^\prime|=|\partial_\theta f|$ on the unit circle, so that the equality condition reduces to saying that $q$ is the conformal pushforward of a constant weight. 

Denote the geometric mean of the quantities $g_0$ and $g_1$ by
\begin{equation} \label{eq:g-underlying}
g = \sqrt{g_0 g_1} \geq 1,
\end{equation}
where $g \geq 1$ by \autoref{le:ggeq1}. Notice $g$ depends on both the mapping $f$ and weight $q$. Write $\R_+=(0,\infty)$ for the positive half-axis. Now we come to the main result, proved in \autoref{sec:underlying}. 
\begin{theorem}[Estimating the Steklov eigenvalues] \label{th:underlying}
Assume $f : \D \rightarrow \Omega$ is a quasiconformal mapping from the disk to a bounded  planar domain, and that $f$ extends to a homeomorphism of the closures with $f : \partial \D \to \partial \Omega$ being absolutely continuous. Suppose the complex dilatation $\mu$ depends only on the angular variable $\theta$, that $q \in L^\infty(\Sigma)$ is a positive weight function on $\Sigma$, and that $p \in L^2[0,2\pi]$. 

Then for each $n \in \N$ and every concave increasing function $C : \R_+ \rightarrow \R$,
\[
  \sum_{j=1}^n C \big( \sigma_j(\Omega,q)L(\Sigma,q) \big) \leq
  \sum_{j=1}^n C \! \left( 2\pi g \left\lceil \frac{j}{2} \right\rceil \right)
\]
with equality if $\Omega$ is a disk, $q \equiv \text{const.}$\ and $f$ is a complex linear map of $\D$ onto $\Omega$.

Equality statement for the first nonzero eigenvalue: if $\sigma_1(\Omega,q)L(\Sigma,q) = 2\pi g$ then $(\Omega,q)$ is conformally equivalent to $(\D,p_c)$ for some constant weight function $p_c$, and equality holds in \autoref{le:ggeq1}. If also $q \equiv1$, then $\Omega$ is a disk.
\end{theorem}
For the first eigenvalue, an old result of Weinstock \cite{W54} says
\begin{equation} \label{eq:weinstock}
\sigma_1(\Omega,q)L(\Sigma,q) \leq 2\pi ,
\end{equation}
which is stronger than \autoref{th:underlying} for $n=1$ since Weinstock does not need the factor $g \geq 1$. The theorem is new for $n \geq 2$. In \autoref{sec:comparison} we will compare with results in the literature, especially the work of Hersch--Payne---Schiffer. Note the sufficient condition for equality in the theorem can be improved using conformal invariance of harmonic functions --- see the sufficient condition for  \autoref{Corollary:SimplyConnectedEstimate}  below. 

Special choices of the concave function $C$ in the preceding theorem yield:
\begin{corollary}\label{Corollary:Special}
Each of the following spectral quantities on $\Omega$ with weight $q$ attains its maximum when $\Omega$ is a disk and $q$ is constant:
  \begin{align*}
    (\sigma_1^s+\cdots+\sigma_n^s)^{1/s}L/g,\qquad\sqrt[n]{\sigma_1\cdots\sigma_n} \, L/g ,
  \end{align*}
where $0 < s \leq 1$. Further, for $s<0<t$ the quantities 
\[
  \sum_{j=1}^n (\sigma_j L/g)^s \qquad \text{and} \qquad
  \sum_{j=1}^n \exp(-t \sigma_j L/g)
\]
are minimal when $\Omega$ is a disk and $q$ is constant.
\end{corollary}
The last two quantities are partial sums of the spectral zeta function and heat trace, respectively, where we have normalized the eigenvalues with $L/g$. 

\subsection*{Simply connected domains and conformal mapping}
Assume $\Omega$ is a simply connected, bounded planar domain with piecewise smooth boundary. The Riemann mapping theorem provides a conformal diffeomorphism
\[
f : \D \rightarrow \Omega .
\]
Because the boundary $\partial\Omega$ is piecewise smooth, the map $f$ extends to a homeomorphism of $\overline{\D}$ onto $\overline{\Omega}$ with $f : \partial \D \to \partial \Omega$ being smooth except at finitely many points. Then $|\partial_\theta f| = |f^\prime|$ on the unit circle, and so the boundary densities are related by 
\[
p=(q \circ f) |f^\prime| .
\]
 
We will need the  geometric quantity
\begin{equation} \label{eq:gammais}
\gamma(\Omega,q)=\frac{\left\{ \left( \frac{1}{2\pi}\int_0^{2\pi} p(\theta)^2 \, d\theta \right)^{\! 2} - \left| \frac{1}{2\pi}\int_0^{2\pi} p(\theta)^2 e^{i\theta} \, d\theta \right|^2 \right\}^{\! 1/4}}{\frac{1}{2\pi} \int_0^{2\pi} p(\theta) \, d\theta} .
\end{equation}
\autoref{Lemma:CenterofMass} shows that the expression on the right side of \eqref{eq:gammais} depends only on $\Omega$ and $q$, and not on the choice of conformal map $f$ through which the weight function $p$ was defined. The lemma will further show that 
\[
\gamma(\Omega,q) \geq 1.
\]
Obviously $\gamma(\Omega,q) = 1$ if $p$ is constant, by \eqref{eq:gammais}. 

Unfortunately, $\gamma$ could equal $+\infty$ or be undefined. For example, when $q$ is constant and $\Sigma$ contains a corner with interior angle $\alpha \pi$, the conformal map $f$ behaves locally like $(z-e^{i\theta_0})^{\alpha}$, and so $p^2 \sim |f^\prime|^2$ is nonintegrable along the circle if $\alpha \leq 1/2$. In particular, this happens when $\Sigma$ has a right angle ($\alpha=1/2$).  To avoid the problem, we will simply assume $p \in L^2[0,2\pi]$, so that $\gamma$ is finite.
\begin{corollary}[Simply connected planar domains] \label{Corollary:SimplyConnectedEstimate}
Assume $\Omega$ is a simply connected, bounded planar domain with piecewise smooth boundary, and consider weights $q$ and $p$ with $0<q \in L^\infty(\Sigma)$ and $p \in L^2[0,2\pi]$, as discussed above. Then for each $n \in \N$ and every concave increasing function $C : \R_+ \rightarrow \R$,
\[
  \sum_{j=1}^n C \big( \sigma_j(\Omega,q)L(\Sigma,q) \big) \leq
  \sum_{j=1}^n C \! \left( 2\pi \gamma(\Omega,q) \left\lceil \frac{j}{2} \right\rceil \right)
\]
with equality when $(\Omega,q)$ is conformally equivalent to $(\D,p_c)$ for some constant weight function $p_c$ (in which case $\gamma(\Omega,q)=1$). 

Equality statement for the first nonzero eigenvalue: if $\sigma_1(\Omega,q)L(\Sigma,q) = 2\pi \gamma(\Omega,q)$ then $(\Omega,q)$ is conformally equivalent to $(\D,p_c)$ for some constant weight function $p_c$. If in addition $q\equiv1$, then $\Omega$ is a disk.
\end{corollary}
The corollary is deduced from \autoref{th:underlying} in \autoref{sec:simplyconn}. The analogue of \autoref{Corollary:Special} holds too, with $g$ replaced by $\gamma(\Omega,q)$ and with the maximum/minimum attained when $(\Omega,q)$ is conformally equivalent to $(\D,p_c)$ for some constant weight function $p_c$. Note that \autoref{Corollary:Special} covers finite sums of the reciprocals of the eigenvalues. This case, for simply connected domains, was already considered by Hersch, Payne and Schiffer \cite[Section 7]{HPS74}. Their lower bounds are stronger than the corresponding cases of our results, because they do not need the factor $g$, but on the other hand their results do not apply to sums or products of eigenvalues.

We illustrate \autoref{Corollary:SimplyConnectedEstimate} with several examples in \autoref{sec:Examples}.

\subsection*{Starlike domains}
A domain in the complex plane is \emph{starlike} if it can be expressed in the form
\[
\Omega = \{ re^{i\theta} : \theta \in [0,2\pi], 0 \leq r < R(\theta) \} 
\]
for some positive, $2\pi$-periodic function $R$ called the \emph{radius function} of $\Omega$. We assume $R$ is Lipschitz continuous. By abusing notation, we write 
\[
q(\theta) = q(R(\theta)e^{i\theta})
\]
for the weight function on $\Sigma = \partial \Omega$. 

\begin{lemma}[Geometric quantities] \label{le:gstarlike}
For the starlike case above, the geometric quantities $g_0,g_1$ defined in \eqref{eq:g01is} are
\begin{align}
  g_0(\Omega) & =1+\frac{1}{2\pi}\int_0^{2\pi}(\log R)'(\theta)^2\,d\theta, \label{eq:g0starlike} \\
  g_1(\Omega,q) & =\frac{\frac{1}{2\pi}\int_0^{2\pi} \big( R(\theta)^2+R^\prime(\theta)^2 \big) q(\theta)^2 \,d\theta}
  {\big(\frac{1}{2\pi}\int_0^{2\pi}\sqrt{R(\theta)^2+R^\prime(\theta)^2} \, q(\theta) \,d\theta\big)^{\! 2}} \, , \label{eq:g1starlike}
\end{align}
and in this case 
\[
g_0 \geq 1, \qquad g_1\geq 1 .
\]
Further, $g_0=1$ if and only if $R$ is constant, which means $\Omega$ is a disk centered at the origin. 
\end{lemma}
Notice that these formulas for $g_0$ and $g_1$ are scale invariant (homogeneous with respect to the radius function $R$).  The equality statement for $g_1$ is more complicated, and we discuss it after the proof of the lemma in \autoref{sec:starlike}

Now we state the result for starlike domains. Here 
\[
g = \sqrt{g_0 g_1}
\]
is the geometric average of the two quantities given in \eqref{eq:g0starlike} and \eqref{eq:g1starlike}. 

\begin{corollary}[Starlike planar domains] \label{Corollary:MainStarlikePlane}
Assume $\Omega$ is a starlike planar domain with Lipschitz continuous radius function, and positive weight function $q \in L^\infty(\Sigma)$. 

Then for each $n \in \N$ and every concave increasing function $C : \R_+ \rightarrow \R$,
\[
  \sum_{j=1}^n C \big( \sigma_j(\Omega,q)L(\Sigma,q) \big) \leq
  \sum_{j=1}^n C \! \left( 2\pi g(\Omega,q) \left\lceil \frac{j}{2} \right\rceil \right)
\]
with equality when $\Omega$ is a disk centered at the origin and $q \equiv \text{const.}$ 

Equality statement for the first nonzero eigenvalue: if $\sigma_1(\Omega,q)L(\Sigma,q) = 2\pi g(\Omega,q) $ then $\Omega$ is a disk centered at the origin and $q \equiv \text{const.}$
\end{corollary}
We show in \autoref{sec:starlike} how to obtain the corollary from \autoref{th:underlying}. Of course, the analogue of \autoref{Corollary:Special} (more general spectral functionals) holds for starlike domains too.

The examples in \autoref{sec:Examples} illustrate the conformal and starlike methods.

\section{\bf Related work in the literature, and comparison with the Hersch--Payne--Schiffer result}
\label{sec:comparison}

\subsection*{Prior work}
This paper proves sharp upper bounds on Steklov eigenvalue sums, under normalization by the perimeter and the additional geometric quantity $g$ or $\gamma$. 

The first geometric upper bound for Steklov eigenvalues is that of Weinstock~\cite{W54}, who proved that among simply connected planar domains of given perimeter, $\sigma_1$ is maximal on a disk. Some years afterward, Hersch--Payne--Schiffer~\cite{HPS74} used a subtle complex analytic method to get bounds on sums of reciprocal eigenvalues and on each individual eigenvalue $\sigma_j$. Their bound on $\sigma_j$ was recently shown to be sharp~\cite{GP10b}. Later in this section, we compare our work with that of Hersch, Payne and Schiffer. 

Uniformization theory enables these results to be generalized to
compact Riemann surfaces with boundary, a setting in which the upper
bounds involve also the number of boundary components and the
genus~\cite{FS11,GP12}. In dimension $\geq 3$, additional geometric
upper bounds have been obtained, as follows. Brock~\cite{B01}
considered domains with fixed volume rather than fixed perimeter, in $\R^n$, and
proved that the ball minimizes the sum of reciprocals $\sum_{j=1}^n
\sigma_j^{-1}$. On compact manifolds, an upper bounds for $\sigma_1$
was given by Fraser--Schoen~\cite{FS11}, in terms of the volume and a quantity which they called the
\emph{relative conformal volume}. For domains, methods from metric geometry were used by Colbois
\emph{et al.}~\cite{CEG11} to bound each individual eigenvalue
$\sigma_j$ in terms of the perimeter and volume of the domain, and this work was recently improved by Hassannezhad~\cite{H11}. For compact hypersurfaces with boundary,
Ilias--Makhoul~\cite{IM11} proved upper bounds for $\sigma_1$ in terms of various
mean curvatures of the boundary.

Turning now to lower bounds, the minimum of each eigenvalue $\sigma_j$
among domains of fixed perimeter or fixed volume is easily seen to be
zero, by a ``pinching'' construction \cite[Section
2.2]{GP10b}. Geometric lower bounds must therefore involve some other restrictions. An early result is
that of Kuttler--Sigillito~\cite{KSi1}, who considered planar starlike
domains and gave a bound in terms of the radius function and its
derivative (see \autoref{sec:starlike}). One should also mention a
recent paper of Jammes~\cite{J14}, where a lower bound in the spirit
of the classical Cheeger inequality is proved for the first nonzero
Steklov eigenvalue. See also \cite{E97}.

Regarding other eigenvalue functionals, Dittmar~\cite{D04} proved that
among simply connected planar domains with given conformal radius, the
disk minimizes the infinite sum of reciprocals of all squares,
$\sum_{j=1}^\infty\sigma_j^{-2}$. Henrot--Philippin--Safoui~\cite{HPS08}
proved that among convex domains of fixed measure in $\R^n$, the
product of the first $n$ nonzero Steklov eigenvalues is maximal for a
ball. Their method is based on an isoperimetric inequality for moment
of inertia. Edward~\cite{E94} proved for simply connected domains $\Omega$ of perimeter $2\pi$ that the relative sum of squares is minimal for the unit disk: $\sum_j \big( \sigma_j(\Omega)^2-\sigma_j(\D)^2 \big) \geq 0$.

Incidentally, to justify the interpretation of the Steklov problem in terms of a membrane whose mass is concentrated at the boundary, one may compare the Rayleigh quotient \eqref{eq:RayStek} for the Steklov problem with the usual Rayleigh quotient for the Neumann Laplacian. For spectral convergence results as the mass concentrates onto the boundary, see recent work of Lamberti and Provenzano \cite{LP14}. 

The literature on the spectral geometry of the Steklov problem is
expanding rapidly, and so we had to omit many papers here. We refer to~\cite{GP10,GP14} for recent surveys.

\subsection*{Comparison with Hersch--Payne--Schiffer result}
For simply connected planar domains, the Hersch--Payne--Schiffer (HPS) inequality \cite{HPS74} 
states that each individual Steklov eigenvalue is bounded according to
\begin{equation}\label{HPSoriginal}
\sigma_j L\leq 2\pi j. 
\end{equation}
Taking $j=1$ recovers Weinstock's inequality \eqref{eq:weinstock}. Equality in \eqref{HPSoriginal} is approached by a sequence of domains tending to a disjoint union of identical disks, as Girouard and Polterovich \cite{GP10} later showed. Summing the HPS inequality leads to
\begin{equation}\label{HPS}
\sum_{j=1}^n\sigma_j L\leq 2\pi \sum_{j=1}^n j=\pi n(n+1) .
\end{equation}
This ``summed HPS'' bound is \emph{not} expected to be sharp, since the original HPS inequality \eqref{HPSoriginal} has a different optimizing domain for each $j$. Thus one can hope that our estimates in \autoref{Corollary:SimplyConnectedEstimate} and \autoref{Corollary:MainStarlikePlane} improve on the summed HPS bound. 

To compare, notice \autoref{Corollary:SimplyConnectedEstimate} implies that 
\begin{equation} \label{eq:HPScomp}
\sum_{j=1}^n\sigma_j L
\leq 2\pi \gamma \sum_{j=1}^n\left\lceil\frac{j}{2}\right\rceil 
= \frac{\pi}{2} \gamma 
\begin{cases}
n (n+2) & \text{if $n$ is even,} \\
(n+1)^2 & \text{if $n$ is odd.} 
\end{cases}
\end{equation}
Hence our bound \eqref{eq:HPScomp} improves on the HPS sum inequality \eqref{HPS} if the geometric factor $\gamma$ satisfies
\[
\gamma < 2 \frac{n+1}{n+2} \qquad \text{($n$ even)}
\]
or 
\[
\gamma < 2 \frac{n}{n+1} \qquad \text{($n$ odd).}
\]
In particular, our bound improves on the HPS sum inequality for all $n \geq 2$ if $\gamma < 3/2$, and improves on it for all large $n$ if $\gamma < 2$. For $n=1$ the HPS--Weinstock bound is  always better, since it does not involve the factor $\gamma \geq 1$. For a starlike domain one obtains the same criteria except with $g$ instead of $\gamma$, by \autoref{Corollary:MainStarlikePlane}. 

Thus for domains close to a disk, our bounds are better by a factor of between $3/2$ and $2$, because $\gamma$ and $g$ are close to $1$ in that case. \autoref{examples} provides more detailed information for some example domains, which are not necessarily close to a disk. 

\subsection*{How our method differs from that of Hersch--Payne--Schiffer} Our trial function method is related to that of Hersch, Payne and Schiffer \cite{HPS74} for the ``conformal'' case (\autoref{Corollary:SimplyConnectedEstimate}), except with a crucial interchange in the order of operations. They proceed as follows: take a Steklov eigenfunction on the unit circle, pre-compose with a uniformization of the circle to push $p$ forward to a constant density,  harmonically extend this composition to the unit disk, and pre-compose the resulting harmonic function with a conformal map from $\Omega$ to the disk, thus obtaining a harmonic trial function on $\Omega$. Note the harmonic extension is performed after the uniformization step. In contrast, in this paper the harmonic extension is carried out before the uniformization step. In other words, we uniformize the Steklov eigenfunction on the whole disk, not just on the circle; see \autoref{sec:underlying} and \autoref{sec:simplyconn} for details.

Both methods preserve length measure on the boundary and hence preserve orthogonality of trial functions (see \cite[p.~101]{HPS74}), but our harmonic extension is easier to work with since it is explicit ($\cos k\theta$ extends harmonically to $r^k \cos k\theta$) and does not get tangled up with the uniformizing map. Further, the method of Hersch \emph{et al.}\ relies very much on conformal invariance of the Dirichlet integral, whereas our approach can handle non-invariance of the Dirichlet integral due to certain quasiconformal maps (see \autoref{sec:underlying}).

\section{\bf Open problems}
\label{sec:problems}

Our theorems maximize Steklov eigenvalue sums with the help of the geometric factor $\gamma$ or $g$. What can one say in the absence of that factor, in other words, if one normalizes solely by perimeter?

For the first eigenvalue, Weinstock's theorem \eqref{eq:weinstock} says $\sigma_1 L$ is maximal for the disk among all simply connected domains. For the second eigenvalue, maximality of $\sigma_2 L$ for the double-disk (in a limiting sense) was proved by Hersch, Payne and Schiffer \cite[formula (3$''$)]{HPS74} and Girouard and Polterovich \cite[\S1.3]{GP10}. 

We ask: 
\[
\text{what domain maximizes $(\sigma_1+\sigma_2)L$?}
\]
The maximizer is certainly not the disk, since an ellipse can give a larger value (see \autoref{tab:ellipse} later in the paper, which shows ``$\rho_2>1$'' for certain ellipses). Thus the disk does not maximize the \emph{arithmetic} mean $\frac{1}{2} (\sigma_1 + \sigma_2)L$. Interestingly, the disk does maximize the \emph{harmonic} mean of those first two eigenvalues among simply connected domains, by Hersch and Payne's extension of Weinstock's method \cite{HP68,W54}. The disk even maximizes the \emph{geometric} mean $\sqrt{\sigma_1 \sigma_2} L$ by Hersch, Payne and Schiffer \cite[formula (1)]{HPS74}. 

Numerical investigations suggest that the extremizer for $(\sigma_1 + \sigma_2)L$ must be somewhat elongated, possibly stadium-like. We remark on a qualitative similarity with the optimal shape for the second Dirichlet eigenvalue, investigated by Bucur-Buttazzo-Henrot \cite{BBH09} and Henrot-Oudet \cite{HO01}, which is stadium-like but not a stadium.

More generally, one would like to maximize the sum $(\sigma_1 + \cdots + \sigma_n)L$. Numerical experiments suggest an appealing open problem for equilateral triangles: to prove that the eigenvalue sum is bounded above by the corresponding sum for the disk. 

See the examples in \autoref{sec:Examples} for more information on eigenvalue sums.

\subsection*{Higher dimensions}
We have not extended our conformal and quasiconformal mapping results to higher dimensions, because the Dirichlet integral fails to transform nicely under such maps. The starlike special case can be extended to higher dimensions, but the resulting geometric quantity is considerably more complicated than in $2$ dimensions, and so we omit these results.

\section{\bf Quasiconformal mapping and angular uniformization --- Proof of \autoref{th:underlying}}
\label{sec:underlying}

We establish some lemmas and then prove the theorem. 
\begin{proof}[Proof of \autoref{le:ggeq1}]
Notice 
\begin{align*}
a_0 a_1 
& = \left( \frac{|e^{4i\theta}-\mu^2|}{1-|\mu|^2} \right)^{\! 2} \geq 1
\end{align*}
by the triangle inequality. An application of Cauchy--Schwarz now gives $g_0 g_1 \geq 1$. 

If $g_0 g_1 = 1$, then the preceding argument implies that $a_0 a_1 = 1$ a.e.\ (using here that $p>0$). Hence the equality condition for the triangle inequality requires that $\arg (\mu^2) = 4\theta$ a.e., which implies $\mu=\pm |\mu| e^{2i\theta}$. Thus $e^{-2i\theta} \mu$ is real, and it lies between $-1$ and $1$, by assumption on $\mu$.

Further, the equality condition for Cauchy--Schwarz implies $c\sqrt{a_0} = \sqrt{a_1 p^2}$ a.e.\ for some constant $c>0$. Since $a_0 a_1=1$ we deduce $ca_0 = p$ a.e., which says
\[
c \frac{(1-e^{-2i\theta}\mu)^2}{1-(e^{-2i\theta} \mu)^2} = (q \circ f) |\partial_\theta f| ,
\]
where we used that $e^{-2i\theta} \mu \in (-1,1)$. After employing the polar identities 
\[
\partial_\theta f = ire^{i\theta}  (1 - e^{-2i\theta} \mu) \partial f, \qquad \partial_r f = e^{i\theta} (1 + e^{-2i\theta} \mu ) \partial f,
\]
in this last equation (and putting $r=1$), we obtain the condition $(q \circ f) |\partial_r f| = c$. 

Reversing the argument shows that the necessary conditions for equality are also sufficient. 

\emph{Comment.} It would be interesting to show that the quantity $g$ controls the deviation of the domain from the disk, that is, to show that if $g$ is close to $1$, then $(\Omega,p)$ must be close to a disk with constant weight function, after a suitable conformal mapping. 
\end{proof}

Next we need a transformation property of the Dirichlet integral under a quasiconformal mapping. 
\begin{lemma} \label{le:Dirichlet_qc}
Suppose $f : \D \rightarrow \Omega$ is a quasiconformal mapping from the disk to a planar domain $\Omega$, and that the complex dilatation $\mu$ depends only on the angular variable $\theta$. Given a real-valued function $h \in H^1 \cap L^\infty_{loc}(\D)$, the Dirichlet integral of $h \circ f^{-1}$ on $\Omega$ can be evaluated over the disk in the following polar form:
\begin{equation} \label{eq:transform}
\int_\Omega |\nabla (h \circ f^{-1})|^2 \, dA = 
\int_\D \Big\{ a_0 h_r^2 + a_1 \frac{h_\theta^2}{r^2} + a_2 h_r \frac{h_\theta}{r} \Big\} \, r dr d\theta 
\end{equation}
where $a_0$ and $a_1$ were defined in \eqref{eq:a0a1are}, and 
\begin{equation} \label{eq:a2is}
a_2 = \frac{2 \Imag \overline{(e^{2i\theta}+\mu(e^{i\theta}))} (e^{2i\theta}-\mu(e^{i\theta}))}{1-|\mu(e^{i\theta})|^2} .
\end{equation}
\end{lemma}
Note that $a_0,a_1,a_2$ are all bounded, since by definition the complex dilatation of a quasiconformal map satisfies $\lVert \mu \rVert_{L^\infty(\D)} < 1$. We have assumed in this lemma that $\mu$ depends only on $\theta$, but that is simply to remain consistent with the rest of the paper; in fact, the lemma and its proof hold without that assumption. 
\begin{proof}
First assume $h \in C^1(\overline{\D})$. We will show
\begin{equation} \label{eq:calculus}
|\nabla (h \circ f^{-1})|^2 \circ f = \Big| - f_r \frac{h_\theta}{r} + \frac{f_\theta}{r} h_r \Big|^2 \big/ J(f)^2
\end{equation}
a.e.\ in $\D$, where $J(f)$ denotes the Jacobian determinant of $f$. (We may differentiate pointwise, since quasiconformal mappings are differentiable a.e.) Indeed, by the chain rule, 
\begin{align*}
|\nabla (h \circ f^{-1})|^2 \circ f
& = \big| (\nabla h) D(f^{-1}) \circ f \big|^2 \\
& = \big| (\nabla h) (Df)^{-1} \big|^2 \\
& = \left| 
\begin{pmatrix} h_x & h_y \end{pmatrix} 
\begin{pmatrix} \ c_y & -b_y \\ -c_x & \ b_x \end{pmatrix} 
\right|^2 \Big/ J(f)^2
\end{align*}
where we have written $f=b+ic$ so that $f$ has Jacobian matrix
\[
Df = 
\begin{pmatrix} b_x & b_y \\ c_x & c_y \end{pmatrix} .
\]
The last formula can be rewritten as 
\begin{align*}
|\nabla (h \circ f^{-1})|^2 \circ f
& = \big|
(-h_y,h_x) \cdot 
(b_x + i c_x,b_y + ic_y) 
\big|^2 \big/ J(f)^2 \\
& = |(\nabla h)U \cdot \nabla f|^2 \big/ J(f)^2
\end{align*}
where $U= \left( \begin{smallmatrix} 0 & 1 \\ -1 & 0 \end{smallmatrix} \right)$ represents rotation by $\pi/2$. Expressing the last two gradient vectors in polar coordinates, we have
\begin{align*}
|\nabla (h \circ f^{-1})|^2 \circ f
& = \big| (-r^{-1} h_\theta \vec{e}_r + h_r \vec{e}_\theta) \cdot (f_r \vec{e}_r + r^{-1} f_\theta \vec{e}_\theta) \big|^2 \big/ J(f)^2 ,
\end{align*}
where $\vec{e}_r$ and $\vec{e}_\theta$ are the unit vectors in the radial and angular directions. Now \eqref{eq:calculus} follows immediately. 

To prove formula \eqref{eq:transform}, we multiply \eqref{eq:calculus} by $J(f)$ and integrate over $\D$ to find
\begin{align*}
\int_\Omega |\nabla (h \circ f^{-1})|^2 \, dA
& = \int_\D \Big| - f_r \frac{h_\theta}{r} + \frac{f_\theta}{r} h_r \Big|^2 \big/ J(f) \, dA \\
& = \int_\D \Big\{ \Big| \frac{f_\theta}{r} \Big|^2 h_r^2 + |f_r|^2 \frac{h_\theta^2}{r^2} - 2 \Real \Big( \overline{f_r} \frac{f_\theta}{r} \Big) h_r \frac{h_\theta}{r}\Big\} \big/ J(f) \, dA .
\end{align*}
Note the Jacobian determinant can be expressed in terms of Wirtinger derivatives as 
\[
J(f) = |\partial f|^2 - |\overline{\partial} f|^2 ,
\]
while the polar derivatives can be expressed as 
\[
f_r = e^{i\theta} \partial f + e^{-i\theta} \overline{\partial} f , \qquad \frac{f_\theta}{ir} = e^{i\theta} \partial f - e^{-i\theta} \overline{\partial} f .
\]
Substituting these expressions, we find
\[
\int_\Omega |\nabla (h \circ f^{-1})|^2 \, dA
= \int_\D \Big\{ \frac{|e^{2i\theta} \partial f - \overline{\partial} f |^2}{|\partial f|^2 - |\overline{\partial} f|^2} h_r^2 + \frac{|e^{2i\theta} \partial f + \overline{\partial} f |^2}{|\partial f|^2 - |\overline{\partial} f|^2} \frac{h_\theta^2}{r^2} - \frac{2 \Real \Big( \overline{f_r} \frac{f_\theta}{r} \Big)}{|\partial f|^2 - |\overline{\partial} f|^2} h_r \frac{h_\theta}{r}\Big\} \, dA .
\]
The coefficients of $h_r^2$ and $h_\theta^2/r^2$ equal $a_0$ and $a_1$, respectively, after dividing the top and bottom lines by $|\partial f|^2$, and similarly the coefficient of the mixed term equals $a_2$. That completes the proof of \eqref{eq:transform}, when $h \in C^1(\overline{\D})$ and the gradient $\nabla(h \circ f^{-1})$ is evaluated pointwise a.e. This pointwise gradient is also the weak gradient, since quasiconformal mappings are absolutely continuous on lines, and so \eqref{eq:transform} holds also in terms of weak derivatives. Hence in particular $\nabla (h \circ f^{-1}) \in L^2(\Omega)$.

To extend \eqref{eq:transform} to the general case of $h \in H^1 \cap L^\infty_{loc}(\D)$, one argues using the density of $C^1(\overline{\D})$ in the Sobolev space. Local boundedness of $h$ in the disk is used to insure local integrability of $h \circ f^{-1}$ in $\Omega$, so that the weak derivative may be defined. 
\end{proof}

Now we can prove the main result. 
\begin{proof}[Proof of \autoref{th:underlying}]
The Steklov spectrum on $\Omega$ has Rayleigh quotient
\begin{equation} \label{eq:RayStek}
\Ray[v]=\frac{\int_\Omega |\nabla v|^2 \, dA}{\int_\Sigma v^2 q \, ds} , \qquad v \in H^1(\Omega).
\end{equation}
The bulk of the proof consists of computing and averaging this Rayleigh quotient for a family of trial functions that we transplant from the disk to $\Omega$ via the given map $f$. Then at the end we put this result into a Rayleigh principle and hence estimate the Steklov eigenvalue sums on $\Omega$.

The weight function $p$ on the circle $\Sp^1$ has total mass $L=\int_0^{2\pi} p \, d\theta = L(\Sigma, q)$. We ``uniformize'' the weight function by means of the map 
\[
\Theta(\theta)=\frac{2\pi}{L}\int_0^\theta p(\eta) \, d\eta ,
\]
with the point being that $\Theta(\theta)$ increases continuously from $0$ to $2\pi$ as $\theta$ increases from $0$ to $2\pi$. In other words, $\frac{2\pi}{L} p \, d\theta$ pushes forward under the map $\Theta$ to arclength measure on the circle. Note that $\Theta^\prime = 2\pi p/L$.

Consider a function $u \in C^1(\overline{\D})$ that is not identically zero on the unit circle. Take an arbitrary $t>0, \phi \in [0,2\pi]$, and fix a choice of $\pm$ sign. Let a function $h$ in polar coordinates be given by
\[
h(r,\theta) = u(r^t,\phi \pm \Theta(\theta)) ,
\] 
and define a trial function on $\Omega$ by
\[
v^{t,\phi, \pm} = h \circ f^{-1} .
\]
(Note how $u$ transforms to $h$ by radial stretching, rotation, possibly reflection, and angular uniformization, and then $f^{-1}$ carries the function from the disk to $\Omega$.) 

Obviously $v^{t,\phi, \pm}$ is continuous and bounded on $\overline{\Omega}$, since $h$ is continuous and bounded on $\overline{\D}$ and $f$ is a homeomorphism of $\overline{\D}$ onto $\overline{\Omega}$. We want to show 
\begin{equation} \label{eq:trialfn}
v^{t,\phi, \pm} \in H^1(\Omega)
\end{equation}
so that this function is a valid trial function. It suffices to show $h \in H^1(\D)$, because then \autoref{le:Dirichlet_qc} applies. First notice $h_r \in L^2(\D)$ because boundedness of $u_r$ implies
\[
\int_\D h_r^2 \, dA \leq (\text{const.}) \int_0^1 (tr^{t-1})^2 \, r dr < \infty .
\]
Second, $r^{-1} h_\theta \in L^2(\D)$ because boundedness of $r^{-1} u_\theta$ implies
\begin{align*}
\int_\D r^{-2} h_\theta^2 \, dA 
& \leq (\text{const.}) \int_0^1 (r^{-1+t})^2 \, rdr \int_0^{2\pi} \Theta^\prime(\theta)^2 \, d\theta \\
& \leq (\text{const.}) \int_0^{2\pi} p(\theta)^2 \, d\theta < \infty 
\end{align*}
since $p \in L^2[0,2\pi]$ by hypothesis. This finishes the proof of \eqref{eq:trialfn}. 

Now we may compute
\begin{align}
\Ray[v^{t,\phi, \pm}]
& = \frac{\int_\Omega |\nabla (h \circ f^{-1})|^2 \, dA}{\int_\Sigma (h \circ f^{-1})^2 q \, ds} \notag \\
& = \frac{\int_\D \big\{ a_0(\theta) h_r^2 + a_1(\theta) r^{-2}  h_\theta^2 + a_2(\theta,r) r^{-1} h_\theta h_r \big\} \, r dr d\theta}{\int_{\Sp^1} h^2 p \, d\theta} \label{eq:rayh} 
\end{align}
by \autoref{le:Dirichlet_qc} and recalling that $\mu,a_0,a_1$ depend only on $\theta$, by the hypotheses of \autoref{th:underlying}. Upon substituting the definition of $h$, we find the denominator equals
\begin{align*}
\int_{\Sp^1} h^2 p \, d\theta & = \int_0^{2\pi} u(1,\phi \pm \Theta(\theta))^2 \frac{L}{2\pi} \Theta^\prime(\theta) \, d\theta \\
& = \frac{L}{2\pi} \int_0^{2\pi} u(1,\phi \pm \Theta)^2 \, d\Theta \\
& = \frac{L}{2\pi} \int_0^{2\pi} u(1,\Theta)^2 \, d\Theta .
\end{align*}
Similarly, the numerator equals 
\begin{align*}
& \int_\D \Big( a_0(\theta) u_r(r^t,\phi \pm \Theta(\theta))^2 t^2 r^{2t-2} + a_1(\theta) r^{-2}  u_\theta(r^t,\phi \pm \Theta(\theta))^2 \Theta^\prime(\theta)^2 \\
& \qquad \qquad  \pm a_2(\theta,r) r^{-1} u_\theta(r^t,\phi \pm \Theta(\theta)) \Theta^\prime(\theta) u_r(r^t,\phi \pm \Theta(\theta)) tr^{t-1} \Big) \, r dr d\theta ,
\end{align*}
which simplifies by the change of variable $r \mapsto r^{1/t}$ to give
\begin{align*}
& \int_\D \Big( t a_0(\theta) u_r(r,\phi \pm \Theta(\theta))^2 + \frac{1}{t} a_1(\theta) \Theta^\prime(\theta)^2 r^{-2}  u_\theta(r,\phi \pm \Theta(\theta))^2 \\
& \qquad \qquad  \pm a_2(\theta,r^{1/t}) \Theta^\prime(\theta) r^{-1} u_\theta(r,\phi \pm \Theta(\theta)) u_r(r,\phi \pm \Theta(\theta)) \Big) \, r dr d\theta .
\end{align*}

By averaging this last expression with respect to $\phi \in [0,2\pi]$, we can separate the $\phi$- and $\theta$-integrals and hence obtain from \eqref{eq:rayh} and the definition of $g_0$ and $g_1$ in \eqref{eq:g01is} that
\begin{align*}
\frac{1}{2\pi} \int_0^{2\pi} \Ray[v^{t,\phi, \pm}] \, d\phi & =  \frac{\int_\D \big( t g_0 u_r(r,\phi)^2 + \frac{1}{t} g_1 r^{-2} u_\phi(r,\phi)^2 \big) \, r dr d\phi}{\frac{L}{2\pi} \int_0^{2\pi} u(1,\phi)^2 \, d\phi} 
 \\
& \pm \frac{\frac{1}{2\pi} \int_0^{2\pi}  \int_0^{2\pi} \int_0^1 a_2(\theta,r^{1/t}) \Theta^\prime(\theta) u_r(r,\phi) u_\phi(r,\phi) \, dr d\phi d\theta}{\frac{L}{2\pi} \int_0^{2\pi} u(1,\phi)^2 \, d\phi} .
\end{align*}
The last term cancels if we average over the choice of $\pm$ sign, and so 
\[
\frac{1}{2} \sum_\pm \frac{1}{2\pi} \int_0^{2\pi} \Ray[v^{t,\phi, \pm}] \, d\phi = \frac{2\pi}{L} \frac{\int_\D \big( t g_0 u_r(r,\phi)^2 + \frac{1}{t} g_1 r^{-2} u_\phi(r,\phi)^2 \big) \, r dr d\phi}{\int_0^{2\pi} u(1,\phi)^2 \, d\phi} .
\]
Making the particular choice $t=\sqrt{g_1/g_0}$ gives 
\[
tg_0 = \frac{1}{t}g_1 = \sqrt{g_0 g_1} = g ,
\]
and so the coefficients in the last formula agree and the numerator reduces to $g$ times the Dirichlet integral of $u$. Thus
\begin{equation} \label{eq:averagedRay}
\frac{1}{2} \sum_\pm \frac{1}{2\pi} \int_0^{2\pi} \Ray[v^{t,\phi, \pm}] \, d\phi = \frac{2\pi g}{L} \Ray[u] .
\end{equation}

Now we apply the above formulas to prove the Theorem. Recall that the sum of the first $n$ nonzero Steklov eigenvalues is characterized by a Rayleigh--Poincar\'{e} Variational Principle \cite[p.~98]{B80}:
\begin{align*}
& \sigma_1 + \dots + \sigma_n = \min \big\{ \Ray[v_1] + \dots + \Ray[v_n] :  v_1, \dots,v_n \in H^1(\Omega) \text{\ are pairwise} \\
& \qquad \qquad \qquad \qquad \quad  \text{orthogonal in $L^2(\Sigma,q \, ds)$ and have mean value zero w.r.t.\ $q \, ds$} \big\} .
\end{align*}
Thus in order to get an upper bound on the eigenvalue sum, we need trial functions $v_1,\ldots,v_n$ satisfying the desired orthogonality properties. 

We start by taking eigenfunctions $u_1,u_2,u_3,\ldots$ for the Steklov problem on the unit disk (having constant weight $1$ on the unit circle), with corresponding eigenvalues $\sigma_j(\D), j=1,2,3,\ldots$ as in \eqref{eq:diskeigenfns}. Note that $u_1, \dots,u_n \in C^1(\overline{\D})$ and these functions are pairwise orthogonal in $L^2(\Sp^1)$ and have mean value zero over $\Sp^1$. Then we construct  trial functions $v_1^{t,\phi, \pm},\ldots,v_n^{t,\phi, \pm}$ by following the method in the proof above. These trial functions belong to $H^1(\Omega)$, and are pairwise orthogonal in $L^2(\Sigma, q \, ds)$ because
\begin{align*}
\int_\Sigma v_l^{t,\phi, \pm} \, v_m^{t,\phi, \pm} \, q \, ds 
& = \int_{\Sp^1} h_l h_m p \, d\theta \qquad \text{recalling that $f$ pushes $p \, d\theta$ forward to $q \, ds$} \\
& = \int_0^{2\pi} u_l(1,\phi \pm \Theta(\theta)) u_m(1,\phi \pm \Theta(\theta)) \frac{L}{2\pi} \Theta^\prime(\theta) \, d\theta \\
& = \frac{L}{2\pi} \int_0^{2\pi} u_l(1,\phi \pm \Theta) u_m(1,\phi \pm \Theta) \, d\Theta \qquad \text{by changing variable} \\
& = 0
\end{align*}
if $l \neq m$, by the pairwise orthogonality of $u_1,\ldots,u_n$. A similar calculation confirms that each trial function $v_l^{t,\phi, \pm}$ has mean value zero with respect to $q \, ds$.

Inserting these trial functions into the Rayleigh--Poincar\'{e} Variational Principle implies that
\begin{equation} 
\label{eq:trialprinc}
\sum_{j=1}^n \sigma_j(\Omega,q) \leq \sum_{j=1}^n \Ray[v_j^{t,\phi, \pm}] .
\end{equation}
The left side of this inequality is independent of the angle $\phi \in [0,2\pi]$ and of the choice of $\pm$ sign that we made in constructing the trial functions. Hence we may average over those quantities, obtaining with the help of \eqref{eq:averagedRay} that
\begin{align*}
\sum_{j=1}^n \sigma_j(\Omega,q) 
& \leq \sum_{j=1}^n \frac{1}{2} \sum_\pm \frac{1}{2\pi} \int_0^{2\pi} \Ray[v_j^{t,\phi, \pm}] \, d\phi \\
& =  \sum_{j=1}^n \frac{2\pi g}{L} \sigma_j(\D) .
\end{align*}
Recall that $\sigma_j(\D)=\lceil j/2 \rceil$ by \eqref{eq:diskspectrum}. Thus multiplying the last equation by $L$ proves the Theorem in the special case where $C$ is the identity function. This special case implies the Theorem for arbitrary concave increasing $C$, thanks to Hardy--Littlewood--P\'{o}lya majorization \cite[{\S}3.17]{HLP88}. (For more references on majorization, see \cite[Appendix~A]{LS13}.)

Equality holds in the theorem if $\Omega$ is a disk of radius $R$ with $q \equiv \text{const.}$\ and $f$ maps $\D$ to $\Omega$ by a complex linear map (dilation, rotation and translation), because in that case we compute $g=1$ while the eigenfunction $r^k \cos(k \theta)$ on $\Omega$ has eigenvalue $\sigma=k/Rq$. Multiplying this eigenvalue by the weighted perimeter $L=2\pi R q$ yields $2\pi k$, which is the quantity appearing on the right side of the inequality in the theorem.
\end{proof}

\begin{proof}[Equality statement for first nonzero eigenvalue]
Assume $\sigma_1(\Omega,q) L(\Sigma,q) = 2\pi g$. Since $g \geq 1$ by \autoref{le:ggeq1}, we conclude $g=1$ by Weinstock's inequality \eqref{eq:weinstock}. That is, equality holds in \autoref{le:ggeq1}.

Further, Weinstock's equality statement \cite[(4.6)]{W54} provides a conformal map from the unit disk to $\Omega$ that pushes a constant weight forward to $q$. If in addition $q \equiv1$, then $\Omega$ is a disk by the argument at the end of \autoref{sec:simplyconn}.

Let us now sketch a direct proof for the equality statement in the theorem, a proof that relies on our method rather than Weinstock's (although still using that $g=1$ by his result). Enforcing equality in the proof of \autoref{th:underlying} shows that equality must hold in \eqref{eq:trialprinc} when $n=1$, for each choice of first eigenfunction $u_1$ on the unit disk and each choice of $\phi,\pm$. Fix $\phi=0$ and choose the ``$+$'' sign. Define a homeomorphism $\Phi : \Omega \to \D$ by 
\[
\Phi = \alpha \circ f^{-1}
\]
where $\alpha : \D \to \D$ is the angular uniformization map $\alpha : (r,\theta) \mapsto (r^t,\Theta(\theta))$. The trial function in the proof above is $v^{t,0,+} = u_1 \circ \Phi$. Its Rayleigh quotient equals $\sigma_1(\Omega,q)$, by equality in \eqref{eq:trialprinc}, and so the trial function must be a Steklov eigenfunction corresponding to $\sigma_1$. If we choose $u_1=r \cos \theta = x_1$, then the trial function $u_1 \circ \Phi$ is simply the first component of the map $\Phi$. Choosing $u_1=r \sin \theta = x_2$ gives the second component of $\Phi$. Hence the components of $\Phi$ are Steklov eigenfunctions, and so are harmonic functions. 

We will show $\Phi$ satisfies the Cauchy--Riemann equation. Since both components of $\Phi$ are Steklov eigenfunctions for $\sigma_1$, we have the boundary condition
\[
\frac{\partial \Phi}{\partial n} = \sigma_1 q \Phi \qquad \text{on $\Sigma$.}
\]
Also, our construction guarantees that $\Phi$ pushes the density $q$ on $\Sigma$ forward to the constant density $L/2\pi$ on the unit circle, meaning $q \, ds = (L/2\pi) \, d\theta$. That is, 
\[
\left\lvert \frac{\partial \Phi}{\partial s} \right\rvert = \frac{2\pi}{L} q = \sigma_1 q ,
\]
where we use that $\sigma_1 = 2\pi g/L$ by hypothesis and $g=1$. The tangent vector $\partial \Phi/\partial s$ at location $\Phi$ on the unit circle points in the counterclockwise direction, since $\Phi$ is sense-preserving, and so 
\[
\frac{\partial \Phi}{\partial s} = i \Phi \left\lvert \frac{\partial \Phi}{\partial s} \right\rvert = i \Phi \sigma_1 q = i \frac{\partial \Phi}{\partial n} 
\]
by the Steklov boundary condition above. Consequently $\Phi$ satisfies the Cauchy--Riemann equation $\partial \Phi / \partial x_2 = i \, \partial \Phi/\partial x_1$ on the boundary $\Sigma$. Harmonicity of $\Phi$ and the maximum principle now guarantee the validity of the Cauchy--Riemann equation throughout the domain $\Omega$, as desired. Thus the homeomorphism $\Phi$ is an analytic function and hence a conformal map.

This proof sketch for the equality statement is not quite rigorous, since we have not justified that $\Phi$ possesses a normal derivative at the boundary or that the Steklov boundary condition holds pointwise. One can avoid these technical concerns by working with the weak form of the Steklov eigenfunction equation. 
\end{proof}

\noindent \emph{Remark.} If the complex dilatation $\mu$ depends on both $r$ and $\theta$ then one can still get a result from the proof of \autoref{th:underlying}, by substituting the trigonometric formula $u=r^k \cos k\theta$ or $r^k \sin k\theta$ for the disk eigenfunction. The quantities $g_0$ and $g_1$ then involve integration over the unit disk of $a_0$ and $a_1 p^2$, respectively, multiplied by $2k r^{2k-2} r \, dr d\theta$. Hence the values of $g_0$ and $g_1$ depend on $k$ and thus on the index $j$ of the eigenvalue, which makes the resulting eigenvalue estimates more complicated. For this reason we assume in the theorem that $\mu$ depends only on $\theta$.

\begin{proof}[Proof of \autoref{Corollary:Special}]
By applying \autoref{th:underlying} with the concave increasing function $C(a)=a^s$, where $0 < s \leq 1$, we obtain maximality of $(\sigma_1^s + \cdots + \sigma_n^s)^{1/s} \, L/g$ for the disk with constant weight. Then the limiting case $s \downarrow 0$ suggests we take $C(a)=\log a$, which yields maximality of the disk for the functional
\[
\sum_{j=1}^n \log (\sigma_j L/g) = n \log \Big( \sqrt[n]{\sigma_1 \cdots \sigma_n} \, L/g \Big) .
\]
When $s<0$ we can choose the concave increasing function $C(a)= - a^s$, which leads to minimality of the disk for
$\sum_{j=1}^n (\sigma_j L/g)^s$. Lastly, for $t>0$ we take $C(a)=-e^{-ta}$, which implies minimality of the disk for $\sum_{j=1}^n \exp(-t \sigma_j L/g)$.
\end{proof}

\section{\bf Simply connected domains --- Proof of \autoref{Corollary:SimplyConnectedEstimate}}
\label{sec:simplyconn}

\begin{example}[Simply connected domain $\Omega$]\label{ex:simplyconn} \rm
Let $f : \D \rightarrow \Omega$ be a conformal mapping, where $\Omega$ is bounded with piecewise smooth boundary. Then $\overline{\partial} f \equiv 0$ because $f$ is analytic, and so $\mu \equiv 0$. Thus $\mu$ is obviously independent of $r$. The definitions \eqref{eq:a0a1are} and \eqref{eq:a2is} give
\[
a_0 = 1 , \qquad a_1 = 1 , \qquad a_2 = 0 ,
\]
which is to be expected from conformal invariance of the Dirichlet integral (cf.\ \eqref{eq:transform}). The associated geometric quantities in \eqref{eq:g01is} are then
\[
g_0 = 1 , \qquad
g_1 = \gamma_1(p) \overset{\text{def}}{=} \frac{\frac{1}{2\pi}\int_0^{2\pi} p(\theta)^2 \,d\theta}
  {\big(\frac{1}{2\pi}\int_0^{2\pi} p(\theta) \,d\theta\big)^{\! 2}} \geq 1 .
\]
\qed
\end{example}

This last example and \autoref{th:underlying} together imply the inequality in \autoref{Corollary:SimplyConnectedEstimate}, although with a bigger (\emph{i.e.,} worse) geometric factor than we are aiming for, namely $g=\sqrt{g_0 g_1}=\sqrt{\gamma_1(p)}$. To reduce this $g$ to $\gamma$ we call on part (i) of \autoref{Lemma:CenterofMass} below, which exploits the freedom to precompose our conformal map with a M\"{o}bius automorphism of the disk. To verify the sufficient condition for equality in the corollary, note $\gamma_1(p_c)=1$ and hence $\gamma_*(\Omega,q)=1$ by the definition below, so that $\gamma(\Omega,q)=1$ by \autoref{Lemma:CenterofMass}(i); now use conformal invariance of the Steklov problem to show equality holds in the corollary. Lastly, to prove the equality statement for \autoref{Corollary:SimplyConnectedEstimate}, one argues as follows:  if $\sigma_1(\Omega,q)L(\Sigma,q) = 2\pi \gamma(\Omega,q)$ then Weinstock's inequality \eqref{eq:weinstock} forces $\gamma(\Omega,q)=1$, and so \autoref{Lemma:CenterofMass} part (iii) yields the desired equality statement. Alternatively, we could use the equality statement in \autoref{th:underlying}.

To state the lemma, we minimize $\gamma_1$ over all choices of conformal map: let
\[
\gamma_*(\Omega,q) 
= \inf \big\{ \gamma_1(\widetilde{p}) : \text{$\widetilde{p}$ is a weight arising from $q$ via a conformal map $\widetilde{f} : \D \to \Omega$} \big\} .
\]
Since $\gamma_1 \geq 1$ we know $\gamma_*(\Omega,q) \geq 1$. The next lemma records some useful properties of the conformal geometric factor  $\gamma_\star$.
\begin{lemma} \label{Lemma:CenterofMass}\ 

(i) $\sqrt{\gamma_*(\Omega,q)}=\gamma(\Omega,q)$, where the latter quantity was defined in \eqref{eq:gammais}. 

(ii) The infimum defining $\gamma_*(\Omega,q)$ is attained for precisely one conformal map $\widetilde{f}$ (up to pre-rotations of the disk $\D$), namely the map such that the measure $\widetilde{p}^{\ 2} d\theta$ on the unit circle has center of mass at the origin:
\[
\int_0^{2\pi} e^{i\theta} \; \widetilde{p}(\theta)^2 \, d\theta = 0 .
\]
In particular, if $\Omega$ and $q$ have $k$-fold rotational symmetry for some $k \geq 2$ then the infimum defining $\gamma_*(\Omega,q)$ is attained when $\widetilde{f}(0)=0$.

(iii) If $\gamma_*(\Omega,q) = 1$ then $(\Omega,q)$ is conformally equivalent to $(\D,p_c)$ for some constant weight function $p_c$. In the unweighted case ($q \equiv 1$), if $\gamma_*(\Omega,1) = 1$ then $\Omega$ is a disk. 
\end{lemma}
\begin{proof}[Proof of \autoref{Lemma:CenterofMass}]\ 

Part (i). The original conformal map $f : \D \to \Omega$ can be related to any other conformal map $\widetilde{f} : \D \to \Omega$ by
\[
f(z) = \widetilde{f}(e^{i\phi} M(z))
\]
for some $\phi \in \R$ and a M\"{o}bius automorphism of the disk of the form 
\[
M(z) = \frac{z+\zeta}{1+z\overline{\zeta}} , \qquad |z| \leq 1 ,
\]
where $\zeta \in \D$ is some given point. Write $w=e^{i\phi}M(z)$. The measure $\widetilde{p} \, |dw|$ associated with $\widetilde{f}$ is the push forward of $p \, |dz|$ under $z \mapsto w$, and so 
\[
\widetilde{p}(w) = p(z) \left| \frac{dz}{dw} \right|
\]
when $|z|=|w|=1$; note here we identify $p(\theta)$ with $p(z)$ when $z=e^{i\theta}$, and so on. Hence
\[
\int_{\Sp^1} \widetilde{p}(w) \, |dw| = \int_{\Sp^1} p(z) \, |dz| 
\]
and 
\begin{align}
\int_{\Sp^1} \widetilde{p}(w)^2 \, |dw| 
& = \int_{\Sp^1} p(z)^2 \frac{1}{|dw/dz|} \, |dz| \notag \\
& = \int_{\Sp^1} p(z)^2 \frac{|1+\overline{\zeta}z|^2}{1-|\zeta|^2} \, |dz| \notag \\
& = \frac{A + 2 \Real (\overline{\zeta} B) + |\zeta|^2 A}{1-|\zeta|^2} \label{eq:AB}
\end{align}
where
\[
A = \int_{\Sp^1} p(z)^2 \, |dz| , \qquad B = \int_{\Sp^1} p(z)^2 z \, |dz| .
\]
Thus to evaluate the infimum $\gamma_*(\Omega,q)$, we must minimize expression \eqref{eq:AB} with respect to the choice of $\zeta \in \D$. 

Clearly we should choose $\arg \zeta$ in \eqref{eq:AB} such that $\overline{\zeta} B = - |\zeta B|\leq 0$. Then \eqref{eq:AB} can be written as 
\[
h(t) = A \frac{1- 2|c|t + t^2}{1-t^2} = A \Big( \frac{1+|c|}{1+t} + \frac{1-|c|}{1-t} - 1 \Big) ,
\]
where $t = |\zeta| < 1$ and $c=B/A$ so that $|c| \leq 1$. In fact, $|c|<1$ since $|B|<A$ (which holds because the density $p$ cannot concentrate at a single point). Note that $h$ is strictly convex for $0<t<1$, with $h^\prime(0)=-2|c|A \leq 0$ and $h(t) \rightarrow \infty$ as $t \rightarrow 1$. Hence $h$ has a unique minimum point $t_\text{min} \in [0,1)$, which we can determine by setting $h^\prime(t)=0$ and solving to find
\[
t_\text{min} = \frac{|c|}{1+\sqrt{1-|c|^2}} .
\]
Hence the minimizing point $\zeta$ is
\[
\zeta_\text{min} = \frac{-c}{1+\sqrt{1-|c|^2}} 
\]
(noting that $\overline{\zeta}_\text{min} B \leq 0$ as required). The minimum value of the expression \eqref{eq:AB} equals
\[
h(t_{min}) = \sqrt{A^2 - |B|^2} 
\]
and thus 
\[
\gamma_*(\Omega,q) = \min_{\widetilde{p}} \gamma_1(\widetilde{p}) = \frac{\frac{1}{2\pi} \sqrt{A^2 - |B|^2} }{\big( \frac{1}{2\pi} \int_{\Sp^1} p(z) \, |dz| \big)^2} = \gamma(\Omega,q)^2 ,
\]
by recalling the definition \eqref{eq:gammais} of $\gamma(\Omega,q)$.

\medskip
Part (ii). The task in this part of the lemma is to show that $\int_{\Sp^1} w \widetilde{p}(w)^2 \, |dw| = 0$ for precisely one value of $\zeta \in \D$, and that this value is the minimizing value $\zeta_\text{min}$ found in Part (i). Then the optimal conformal map $\widetilde{f}$ is unique, up to pre-rotation by the angle $\phi$. 

So consider an arbitrary $\zeta \in \D$ and compute (as in the proof of Part (i)) that
\begin{align*}
\int_{\Sp^1} w \widetilde{p}(w)^2 \, |dw|
& = \int_{\Sp^1} p(z)^2 \frac{w}{|dw/dz|} \, |dz| \\
& = \int_{\Sp^1} p(z)^2 \frac{(z+\zeta)(1+\zeta \overline{z})}{1-|\zeta|^2} \, |dz| \, e^{i\phi} \\
& = \frac{B+2A\zeta +\overline{B} \zeta^2}{1-|\zeta|^2} \, e^{i\phi} .
\end{align*}
The numerator of this last expression vanishes if and only if $\zeta=\zeta_\text{min}$, as desired. (The quadratic has a second root, when $B \neq 0$, but that root lies outside the unit disk whereas $\zeta$ must lie inside the disk.)

Suppose now that $\Omega$ and $q$ are invariant under rotation of the plane by angle $2\pi/k$, for some $k \geq 2$. If we choose the conformal map $\widetilde{f}$ to map the origin to the origin, then $\widetilde{f}$ commutes with rotation by angle $2\pi/k$ (meaning $\widetilde{f}(w) = e^{-2\pi i/k} \widetilde{f}(e^{2\pi i/k}w)$ for all $w \in \D$). Hence $\widetilde{p}$ is invariant under rotation by angle $2\pi/k$, and so $\int_{\Sp^1} w \widetilde{p}(w)^2 \, |dw| = 0$. Thus the infimum defining $\gamma_*(\Omega,q)$ is attained when $\widetilde{f}(0)=0$.

\medskip
Part (iii). Suppose $\gamma_*(\Omega,q)=1$ and that $\widetilde{p}$ is the weight that achieves the infimum for $\gamma_*$. Then by definition of $\gamma_*$,
\[
\frac{\frac{1}{2\pi}\int_0^{2\pi} \widetilde{p}(\theta)^2 \, d\theta}{\left(\frac{1}{2\pi}\int_0^{2\pi} \widetilde{p}(\theta) \, d\theta \right)^{\! \! 2}} = \gamma_1(\widetilde{p}) = \gamma_*(\Omega,q) = 1 ,
\]
and so from the equality conditions in Cauchy--Schwarz we deduce that $\widetilde{p}$ is constant. Thus $(\Omega,q)$ is conformally equivalent to $(\D,\text{const})$.

Now suppose $q \equiv 1$ and $\gamma_*(\Omega,1) = 1$. Then by the case we just proved, we know $(\Omega,1)$ is conformally equivalent to $(\D,\text{const})$. Let $f : \D \rightarrow \Omega$ be the conformal equivalence. Then since a constant weight on $\Sp^1$ pushes forward to $q \equiv 1$, we find $|f^\prime|$ is constant on the unit circle. Hence $|f^\prime|$ is constant on the unit disk by the maximum principle applied to the harmonic function $\log |f^\prime|$. Therefore $f^\prime$ itself is constant, and so $f$ is linear and $\Omega$ is a disk. 
\end{proof}

\section{\bf Starlike planar domains --- Proof of \autoref{Corollary:MainStarlikePlane}}
\label{sec:starlike}

\begin{example}[Starlike domain $\Omega$, with radius function $R$] \label{ex:starlike} \rm
Let $f : \D \rightarrow \Omega$ be the stretch homeomorphism defined in polar coordinates by $f(re^{i\theta})=R(\theta) re^{i\theta}$. Then 
\begin{align*}
\overline{\partial} f
& = \frac{e^{i\theta}}{2}(f_r + if_\theta/r) = e^{2i\theta}iR^\prime(\theta)/2 , \\
\partial f
& = \frac{e^{-i\theta}}{2}(f_r - if_\theta/r) = (2R(\theta)-iR^\prime(\theta))/2 , 
\end{align*}
and so the complex dilatation is 
\begin{equation} \label{eq:mu-starlike}
\mu = \frac{\overline{\partial} f}{\partial f} = e^{2i\theta} \frac{iR^\prime(\theta)}{2R(\theta)-iR^\prime(\theta)} .
\end{equation}
Notice $\mu$ depends only on $\theta$, and that $\lVert \mu \rVert_{L^\infty(\D)} < 1$ since $R$ is bounded below away from $0$ and $R^\prime$ is bounded above (recalling $R$ is Lipschitz). Thus $f$ is quasiconformal. 

After substituting $\mu$ into the definitions in \eqref{eq:a0a1are} and \eqref{eq:a2is}, we find 
\[
a_0 = 1 + (\log R)^\prime(\theta)^2, \qquad a_1 = 1 , \qquad a_2 = -2 (\log R)^\prime(\theta) .
\]
Alternatively, one can verify directly that the Dirichlet integral transforms according to \eqref{eq:transform} using these formulas for $a_0,a_1,a_2$, by inserting the starlike stretch mapping $f$ into the left side of \eqref{eq:transform} and evaluating in polar coordinates.

To compute $p$ one needs the formula for arclength density along $\Sigma$, which in polar coordinates says:
\[
\frac{ds}{d\theta} = |\partial_\theta f(e^{i\theta})| = \sqrt{R(\theta)^2+R^\prime(\theta)^2} .
\]
Then the geometric quantities $g_0$ and $g_1$ are found to equal the formulas \eqref{eq:g0starlike} and \eqref{eq:g1starlike}, which proves \autoref{le:gstarlike}. Obviously $g_0 \geq 1$ with equality if and only if $R$ is constant, and by Cauchy--Schwarz, $g_1 \geq 1$ with equality if $R$ is constant.

\smallskip
\noindent \emph{Remark on the equality case in \autoref{le:gstarlike} for $g_1(\Omega,1)$.} Consider the unweighted case $q\equiv 1$. If $g_1=1$ then $R^2 + (R^\prime)^2 = \text{const.}$ by definition \eqref{eq:g1starlike} and Cauchy--Schwarz. This condition certainly holds for the constant function $R$, but there are also other solutions. The hippopede with $\delta=0$ (two tangent circles) provides a different solution, although not simply connected; see \autoref{sec:example_hippopede}. Take a union of two such hippopedes, the second one rotated by 90 degrees, to get a flower-shaped starlike domain with $g_1=1$. More generally, take at least three points on the unit disk centered at the origin, in such a way that their convex hull contains the origin. The union of the unit disks centered at the chosen points is a piecewise smooth, starlike domain with $g_1=1$.\qed
\end{example}

The inequality in \autoref{Corollary:MainStarlikePlane} now follows from \autoref{th:underlying} and the above \autoref{ex:starlike}. 

Suppose equality holds in the theorem for $\sigma_1$. Then Weinstock's inequality \eqref{eq:weinstock} forces $g(\Omega,q)=1$. In particular, $g_0(\Omega)=1$ in \eqref{eq:g0starlike}, forcing $R$ to be constant. Therefore $\Omega$ is a centered disk, and so  $g_1(\Omega,q)=1$ in \eqref{eq:g1starlike}; now Cauchy--Schwarz gives that $q$ is constant. Alternatively, to complete the equality statement one may use the equality statement of \autoref{le:ggeq1} provided the Beltrami equation is assumed to hold on the unit circle: if $g=1$ then $e^{-2i\theta}\mu$ is real and so $R^\prime \equiv 0$ by \eqref{eq:mu-starlike}, and also $(q \circ f) |\partial_r f| = (q \circ f) R$ is constant; hence both $R$ and $q$ are constant.

\begin{remark} \rm
The same starlike stretch mapping $f$ as above was used by Kuttler and Sigillito \cite{KSi1} to find \emph{lower} bounds for Steklov eigenvalues. They pulled $\Omega$ back to a disk, as we do, and got a distorted Rayleigh quotient. They used the minimum value of the distortion factor in order to estimate the Rayleigh quotient from below, and hence obtained a lower bound on eigenvalues. In contrast, our method averages the distortion factor and hence obtains upper bounds.
\end{remark}

\medskip
Let us add some remarks about the existence of an optimal choice of origin, for minimizing the quantities $g_0$ and $g_1$ appearing in the starlike result.

The \emph{Lipschitz kernel} of $\Omega$ is the set $\Omega_\text{ker}$ of all points in $\Omega$ with respect to which $\Omega$ is starlike with Lipschitz radius function. Clearly this kernel is an open set (although that would be false if we dropped the Lipschitz assumption, by the example of a slit disk). 
\begin{lemma}[Minimizing the geometric quantities through an optimal choice of origin] \label{Lemma:StarlikeOrigin}
Assume $\Omega$ is starlike with Lipschitz radius function, and that $q \equiv 1$. 

(i) Then the quantities $g_0$ and $g_1$ given in \eqref{eq:g0starlike} and \eqref{eq:g1starlike}, with $q \equiv 1$, are strictly convex when regarded as functions of an origin point in $\Omega_\text{ker}$. 

(ii) If $\Omega$ has $k$-fold rotational symmetry about a point in $\Omega_\text{ker}$, for some $k \geq 2$, then $g_0$ and $g_1$ are minimized when the origin is taken at that center of symmetry.

(iii) If $\Omega$ is convex then a choice of origin exists in $\Omega$ that minimizes $g=\sqrt{g_0 g_1}$.
\end{lemma}
\begin{proof}[Proof of \autoref{Lemma:StarlikeOrigin}]\ 

Part (i). First we modify our notation to emphasize the dependence of $g_0$ and $g_1$ on the origin point $\omega \in \Omega_\text{ker}$ with respect to which the radius function $R_w$ is defined:
\begin{align*}
  g_0(\omega) & =1+\frac{1}{2\pi}\int_0^{2\pi}(\log R_\omega)'(\theta)^2\,d\theta, \\
  g_1(\omega) & =\frac{\frac{1}{2\pi}\int_0^{2\pi} \big( R_\omega(\theta)^2+R_w^\prime(\theta)^2 \big) \,d\theta}
  {\big(\frac{1}{2\pi}\int_0^{2\pi}\sqrt{R_\omega(\theta)^2+R_\omega^\prime(\theta)^2} \,d\theta\big)^{\! 2}} .
\end{align*}
Here the domain $\Omega$ is fixed and $q \equiv 1$. 

The denominator of $g_1(\omega)$ equals the boundary length $L(\Sigma)$ divided by $2\pi$, which is obviously independent of the choice of origin $\omega$. Thus the task is to prove strict convexity of the numerator term, which is
\begin{align}
\int_0^{2\pi} \big( R_\omega(\theta)^2+R_w^\prime(\theta)^2 \big) \,d\theta 
& = \int_0^{2\pi} \left( \frac{ds}{d\theta} \right)^{\! \! 2} \,d\theta \notag \\
& \text{where angle $\theta$ is measured around the origin at $\omega$,} \notag \\
& = \int_\Sigma \frac{ds}{d\theta} \,ds \notag \\
& = \int_\Sigma \frac{|x-\omega|^2}{(x-\omega) \cdot N(x)} \, ds(x) , \label{eq:omegaconvex}
\end{align}
where the last formula for $ds/d\theta$ follows from a simple geometric analysis (see for example \cite[proof of Lemma 10.2]{LS13}). 

Since an integral of convex functions is convex, for convexity it suffices to fix $x \in \Sigma$ and prove  convexity of the last integrand as a function of $\omega \in \Omega_\text{ker}$. We might as well assume $x=0$ (by translating the domain) and that $N(x)=(-1,0)$ points in the negative horizontal direction (by rotating the domain). Then the task is to prove convexity of the function
\[
K(\omega) = \frac{|\omega|^2}{\omega_1} = \omega_1 + \frac{\omega_2^2}{\omega_1} ,
\]
where $\omega=(\omega_1,\omega_2)$ and we note that $\omega_1 > 0$ by starlikeness of the domain. The Hessian matrix is
\[
D^2 K = \frac{2}{\omega_1^3} 
\begin{pmatrix}
\omega_2^2 & -\omega_1 \omega_2 \\
-\omega_1 \omega_2 & \omega_1^2
\end{pmatrix} ,
\]
which is nonnegative definite. Hence $K$ is convex. 

We must still justify that the integral \eqref{eq:omegaconvex} is \emph{strictly} convex as a function of $\omega$. The Hessian matrix of $K$ has one zero eigenvalue, whose eigenvector (null direction) is $\omega$ itself. Relaying that information back to formula \eqref{eq:omegaconvex}, we see that the Hessian of the integrand (the second derivative matrix with respect to $\omega$) has null direction $x-\omega$. Thus the second directional derivative of \eqref{eq:omegaconvex} at point $\omega$ in an arbitrary direction $y$ is positive, because $y$ cannot be parallel to $x-\omega$ for all $x \in \Sigma$. 

Turning now to $g_0(\omega)$, we observe that  
\[
2\pi g_0(\omega) = \int_0^{2\pi} \big( R_\omega(\theta)^2+R_w^\prime(\theta)^2 \big)/R_\omega(\theta)^2  \,d\theta = \int_\Sigma \frac{1}{(x-\omega) \cdot N(x)} \, ds ,
\]
and so after the same reductions as above, the question reduces to convexity of $1/\omega_1$ in the right half plane, which is obvious. For the strictness of the convexity one argues in a similar fashion to above (details left to the reader). 

In fact, convexity of $g_0(\omega)$ was proved already by Aissen \cite[Section 5, Theorem 3]{A58}, as a corollary of strict subharmonicity of $g_0$.

\medskip
Part (ii). Notice $g_0$ and $g_1$ must have critical points when the origin $\omega$ sits at the center of symmetry, by the convexity in Part (i). The strictness of the convexity then implies that these critical points are global minima. 

\medskip Part (iii). Suppose $\Omega$ is convex, so that it is starlike with respect to any choice of origin inside the domain. We know $g_1(\omega) \geq 1$, and so to show existence of a minimizing $\omega$, we need only show $g_0(\omega)$ blows up as $\omega$ approaches the boundary curve $\Sigma$. This fact was proved by Aissen \cite[Theorem 2]{A58}, and since his argument is short, we present a version of it here. Let $\omega_0 \in \Sigma$. Then by Fatou's Lemma,
\begin{align*}
\liminf_{\omega \rightarrow \omega_0} 2\pi g_0(\omega) 
& \geq \int_\Sigma \frac{1}{(x-\omega_0) \cdot N(x)} \, ds(x) \\
& \geq \int_\Sigma \frac{1}{|x-\omega_0|} \, ds(x) \\
& \geq \int_\Sigma \frac{1}{|s|} \, ds ,
\end{align*}
where we have chosen to measure arclength $s$ on $\Sigma$ starting from the point $\omega_0$ (at which $s=0$). The last integral diverges, and so $g_0$ blows up as $\omega$ approaches $\Sigma$.
\end{proof}

\section{\bf Remarks on composite transformations}\label{sec:compound}

So far we have described two methods of generating quasiconformal maps for which the complex dilatation $\mu$  is purely angular, namely, conformal maps and starlike maps. Can we profit from composing such maps?

In \autoref{Lemma:CenterofMass} we pre-composed a conformal map with a M\"obius transformation to find the best possible conformal map. Let us indicate one way to extend this optimization procedure to the quasiconformal case. Let $f:\D\to\Omega$ be conformal, $\psi:\D\to\Omega$ be quasiconformal with purely angular $\mu$ (\emph{e.g.}\ a starlike map), and $M$ be a M\"{o}bius automorphism of the disk. Define
\begin{align*}
  \Psi = f\circ M \circ f^{-1} \circ \psi.
\end{align*}
Then one easily checks that $\mu_\Psi=\mu_\psi$. Therefore $a_0(\Psi)=a_0(\psi)$ and $a_1(\Psi)=a_1(\psi)$. With $F=f\circ M\circ f^{-1}$ (automorphism of $\Omega$) we have
\begin{align*}
  p_\Psi(\theta)&=|\partial_\theta \Psi(e^{i\theta})| = |\partial \Psi| |e^{2i\theta}-\mu_\psi|=|F'\circ \psi| p_\psi(\theta) .
\end{align*}
The presence of $M$ in the formula for $F$ allows for an origin-optimization reminiscent of the one from \autoref{Lemma:CenterofMass}. (In fact, if $\psi=f$ then we reduce back to that case.) This approach can be used to optimize the choice of origin for a starlike domain, and while the optimization might be theoretically difficult, it should remain feasible numerically.  

Another use of composite maps would be to pull back a Steklov problem from a non-starlike domain to a starlike one through a conformal transformation, and then estimate the eigenvalues on that starlike domain using \autoref{Corollary:MainStarlikePlane}. This two-step procedure might yield a better estimate than a direct application of our conformal result \autoref{Corollary:SimplyConnectedEstimate}.

\section{\bf Examples}\label{examples}
\label{sec:Examples}

How sharp are our theorems when compared with the summed Hersch--Payne--Schiffer bound \eqref{HPS}? Or compared with the actual Steklov eigenvalues? To gain intuition on these questions, we will investigate families of regular polygons, ellipses and hippopedes, applying both our conformal mapping and starlike approaches. 

Take the weight to be constant, $q \equiv 1$, throughout this section. 

We determined (in \autoref{sec:comparison}) conditions on $g$ and $\gamma$ under which our estimates on Steklov eigenvalue sums are stronger than the summed HPS bound $\sum_{j=1}^n \sigma_j L \leq \pi n(n+1)$. For the sake of simplicity, we will concentrate on the sum of the first two eigenvalues, $n=2$. In that case our bounds are better than the summed HPS inequality whenever 
\[
g \leq \frac{3}{2} \qquad \text{or} \qquad \gamma \leq \frac{3}{2} .
\]
For larger $n$, the quantities $g$ and $\gamma$ are allowed to be even larger. 

We also want to compare our bounds with the numerically computed values of $(\sigma_1 +\sigma_2)L$ and longer eigenvalue sums, computed using the Finite Element Method with piecewise linear or quadratic conforming elements. Nonconvex domains (e.g.\ hippopedes in \autoref{sec:example_hippopede}) are rather challenging due to re-entrant corners in the  polygonal approximating domain. Further, the boundary approximation introduces errors even for convex domains. To get more accurate results we used an adaptive mesh refinement method (see Garau Morin \cite{GM09}), based on residual errors, and also a boundary snapping mechanism. We chose FEniCS \cite{LMW} to implement the numerical scheme. The scheme is based on similar ones used for mixed Steklov eigenvalue problems by Kuznetsov \textit{et.al.} \cite{Ketal} and Kulczycki--Kwa\'snicki--Siudeja \cite{KKS}.

For the numerical comparisons, we define the ratio
\begin{align*}
\rho_n &= \rho_n(\Omega) = \frac{\sum_{j=1}^n \sigma_j (\Omega) L(\Sigma)}{\sum_{j=1}^n \sigma_j(\D) \cdot 2\pi },\\
\rho_{max} &= \max_n \rho_n .
\end{align*}
In particular, 
\[
\rho_2 = \frac{(\sigma_1 + \sigma_2)L}{4\pi} .
\]
Notice $\rho_2(\D)=1$ (in fact, $\rho_n(\D)=1$). By our estimate \eqref{eq:HPScomp} for the conformal method, and its analogue for the starlike method, we have 
\[
\rho_2 \leq \rho_{max} \leq \gamma \qquad \text{and} \qquad \rho_2  \leq \rho_{max}\leq g .
\]
If $\rho_2$ is close to $\gamma$ or $g$, for a specific domain, then we conclude that our theorems provide a good estimate on the sum of the first two Steklov eigenvalues, and similarly for $\rho_{max}$ with the sum of arbitrary length. It seems that $\rho_{max}=\rho_2$ in many cases. But somewhat surprisingly, it seems the maximal value $\rho_{max}$ may occur for an arbitrarily large value of $n$.

\subsection{Regular polygons}
\label{sec:example_regular}
For the regular $N$-gon centered at the origin, we collect values of $\rho_{max}, g$ and $\gamma$ in \autoref{tab:regular}. We also indicate which $n$-values give $\rho_{max}$. (These results will be explained below.)

The starlike approach performs better in each case (since $g < \gamma$), and both the starlike and conformal mapping approaches improve on the summed HPS bounds for all $n \geq 2$ (since $g,\gamma < 3/2$) except that the conformal method gives no result for equilaterals or squares ($N=3,4$). 
%
%
Lastly, we see $\gamma$ and $g$ are nearly $1$ for $N \geq 8$, which is to be expected since the $N$-gon is almost circular.

Now we explain how to compute $g$ and $\gamma$. 

\begin{table}[t]
  \centering
  \begin{tabular}{clllllll}
    \toprule
  $N$ & 3 & 4 & 5 & 6 & 8 & 10\\
  \midrule
  $\rho_{max}$ & 1{\tiny\,$n\!=\!\infty$} & 1.0013{\tiny\,$n\!=\!9$}& 1.0097{\tiny$n\!=\!7$}& 1.0061{\tiny$n\!=\!9$} & 1.0016{\tiny$n\!=\!{13}$}&1.0012{\tiny$n\!=\!{15}$} \\
  $g$ & 1.4142   & 1.1547   & 1.0844 & 1.0541 & 1.0282 & 1.0174\\
  $\gamma$ & $\infty$ & $\infty$ & 1.3096 & 1.1374 & 1.0527 & 1.0281 \\
  \bottomrule
  \end{tabular}
 \vspace{6pt} 
 \caption{Regular polygon centered at the origin with $N$ sides: values of the ratio $\rho_{max}$ and constants $g$ and $\gamma$. The starlike method gives better results than summed HPS bounds, since $g<3/2$ in each case. The conformal method also gives reasonable bounds for $5$ sides and higher, since $\gamma<3/2$ in those cases. On equilateral triangles $\rho_{max}$ seems to equal $1$, attained in the limit as $n \to \infty$ (see the open problem in \autoref{sec:problems}). }
  \label{tab:regular}
\end{table}

\subsubsection*{Starlike method}
Due to symmetry of the regular polygons, we only need to define the radius function on $(0,\pi/N)$ and multiply all integrals by $2N$. The regular polygon with inscribed circle of radius $1$ is given by
\begin{align*}
  R(\theta)=\sec\theta .
\end{align*}
Hence
\begin{align*}
  (\log R)^\prime & =\tan \theta , \\
  R^2+(R^\prime)^2 & =\sec^4 \theta .
\end{align*}
Therefore \eqref{eq:g0starlike} and \eqref{eq:g1starlike} for regular polygons give
\begin{align*}
  g_0&=1+\frac{2N}{2\pi}\int_0^{\pi/N} \tan^2 \theta \, d\theta = \frac{N}{\pi}\tan(\pi/N)=\frac{L^2}{4\pi A} , \\
  g_1&=\frac{\frac{2N}{2\pi}\int_0^{\pi/N}\sec^4 \theta \,d\theta}
  {\left(\frac{2N}{2\pi}\int_0^{\pi/N}\sec^2 \theta \,d\theta\right)^{\! 2}}=\frac{\pi}{N} \Big( \frac{1}{3} \tan \big( \frac{\pi}{N} \big)+\cot \big( \frac{\pi}{N} \big) \Big) .
\end{align*}
Note that $g_0$ equals the isoperimetric ratio for the domain. This fact was observed already by Aisssen \cite[Section 3]{A58}, for any polygon with an inscribed circle.

The equations above yield 
\begin{equation} \label{eq:gasymptotic1}
g=\sqrt{g_0 g_1}=\sqrt{1+\frac{1}{3}\tan^2 \frac{\pi}{N}} = 1+\frac{\pi^2}{6N^2}+\frac{7\pi^4}{72N^4}+O \left( \frac{1}{N^6} \right)
\end{equation}
from which the values in \autoref{tab:regular} are computed. The formula confirms our expectation that $g$ should approach $1$ as the number of sides increases to infinity. 

\subsubsection*{Conformal method}
The Schwarz-Christoffel map provides a conformal map $f$ of the unit disk to a regular $N$-gon, with the origin mapping to the center. The map is defined through its derivative
\begin{align*}
  f^\prime(z)=\frac{1}{\sqrt[N]{\left(\frac{1-z^N}{2}\right)^2}}.
\end{align*}
Hence 
\begin{align*}
  p(\theta)&=\frac{1}{\sqrt[N]{\left|\frac{1-e^{iN\theta}}{2}\right|^2}}=\frac{1}{\sqrt[N]{\sin^2(N\theta/2)}} .
\end{align*}
Due to rotational symmetry of the regular polygon, $\int_0^{2\pi} e^{i\theta} p(\theta)^2\,d\theta=0$. Therefore $\gamma(\Omega,1)^2=\gamma_1(p)$ (see \autoref{Lemma:CenterofMass}). 

Symmetry again enables us to reduce integrals to the range $\theta\in (0,\pi/N)$, and the substitution $t=\sin^2(N\theta/2)$ then shows that
\begin{align*}
  \frac{1}{2\pi}\int_{0}^{2\pi} p(\theta)\,d\theta
& =\frac{1}{\pi}B\left(\frac{1}{2}-\frac{1}{N},\frac{1}{2} \right) ,
\end{align*}
where $B(a,b)$ is the beta function. Similarly
\begin{align*}
  \frac{1}{2\pi}\int_{0}^{2\pi} p(\theta)^2\,d\theta=\frac{1}{\pi }
B\left(\frac{1}{2}-\frac{2}{N},\frac{1}{2} \right).
\end{align*}
This last integral diverges for $N=3,4$, and hence the conformal method fails to give a finite bound for equilateral triangles and squares. In fact, the Schwarz--Christoffel map for any domain with an interior angle of $\pi/2$ or smaller will give a weight $p$ that does not belong to $L^2$, and so $\gamma_1$ is infinite in such cases.

Rewriting the beta function using gamma functions, one can show from \eqref{eq:gammais} that
\begin{align*}
  \gamma(\Omega,1) = \sqrt{\gamma_1(p)} 
%
%
%
%
= \frac{\Gamma(1-4/N)^{1/2} \Gamma(1-1/N)^2}{\Gamma(1-2/N)^2}.
\end{align*}
The values of $\gamma$ in \autoref{tab:regular} follow directly. Further, with the help of the series expansion of $\Gamma(1+z)$ we obtain the expansion
\begin{equation} \label{eq:gasymptotic2}
  \gamma(\Omega,1) = 1+\frac{\pi^2}{6N^2} +\frac{6\zeta(3)}{N^3}+\frac{103\pi^4}{360N^4}+O \big( \frac{1}{N^6} \big),
\end{equation}
where $\zeta$ is the Riemann zeta function. Comparing \eqref{eq:gasymptotic1} and \eqref{eq:gasymptotic2}, we see the starlike and conformal methods agree up to the second order. The starlike method is better due to the absence of the cubic term.

\subsection{Ellipses} 
\label{sec:example_ellipse}

\begin{table}[t]
  \centering
  \begin{tabular}{cllllll}
    \toprule
    $\varepsilon^2$ & 0 & 1/4 & 1/2 & 3/4 \tiny{(2:1 ellipse)} & 8/9 \tiny{(3:1 ellipse)} & 99/100 \tiny{(10:1 ellipse)}\\
  \midrule
  $\rho_{max}$ & 1 & 1.0058{\tiny$n\!=\!2$} & 1.0340{\tiny$n\!=\!2$} & 1.1311{\tiny$n\!=\!2$} & 1.0896{\tiny$n\!=\!6$} & 1.1566{\tiny$n\!=\!{14}$} \\
  $g$     & 1 & 1.0065 & 1.0382 & 1.1607 & 1.4448 & 3.9995 \\
  \bottomrule
  \end{tabular}
 \vspace{6pt} 
 \caption{Ellipse centered at the origin with eccentricity $\e$, giving values of the ratio $\rho_{max}$ and constant $g$. The starlike method gives better results than summed HPS bounds on most ellipses, since $g<3/2$ when $\e^2 \leq 8/9$. When the eccentricity is large the starlike method is worse for each $n$ than summed HPS, since $g>2$ when $\e^2 \gtrsim 0.95$.
}
  \label{tab:ellipse}
\end{table}

Ellipses are another natural family of examples. We will apply the starlike method but not the conformal method, since the conformal map from a disk to the interior of an ellipse is rather complicated (involving incomplete elliptic integrals). 


\subsubsection*{Starlike method}
Consider an ellipse centered at the origin with longer semiaxis of length $1$ along the horizontal axis and with eccentricity $\varepsilon$. The perimeter can be expressed using the complete elliptic integral of the second kind, giving $L=4E(\varepsilon)$. 
The radius function of the ellipse is 
\begin{equation} \label{eq:ellipseradius}
  R(\theta)=\frac{\sqrt{1-\varepsilon^2}}{\sqrt{1-\varepsilon^2\cos^2(\theta)}}
\end{equation}
and hence one can compute
\begin{align}
g_0 & = 1+ \frac{1}{2\pi} \int_{0}^{2\pi} (\log R)^\prime(\theta)^2 \, d\theta = \frac{1-\varepsilon^2/2}{\sqrt{1-\varepsilon^2}}, \label{eq:ellipseg0} \\
g_1 & = \frac{\frac{1}{2\pi} \int_0^{2\pi} \big( R^2+(R^\prime)^2 \big) \, d\theta}{(L/2\pi)^2} = \frac{1-\varepsilon^2+\varepsilon^4/8}{\sqrt{1-\varepsilon^2}} \frac{\pi^2}{4E(\e)^2} . \notag
\end{align}
Hence
\[
g=\sqrt{g_0 g_1}=1+\frac{5}{64}\varepsilon^4+ \frac{5}{64} \varepsilon^6+O(\varepsilon^8).
\]
See \autoref{tab:ellipse} for values of $g$ and $\rho_{max}$ for a few values of the eccentricity. For moderate eccentricity we get quite accurate results (meaning $g$ is close to $\rho_{max}$, which equals $\rho_2$). From the table one can also compare our results to the summed HPS bounds, finding that except for highly eccentric ellipses, our bounds are better.

\subsection{Hippopedes}
\label{sec:example_hippopede}

Now we invert ellipses with respect to the unit circle centered at the origin, obtaining the family of curves called \emph{hippopedes}. The family includes stadium-like sets and two slightly overlapping ``almost-circles''. See \autoref{fig:hippo}. Note these curves are $2$-fold symmetric, and so the optimal origin for $g$ is at the center of the domain by \autoref{Lemma:StarlikeOrigin}(ii). \autoref{tab:hippo} summarizes our findings, based on formulas for $g$ and $\gamma$ developed below.

\def\X[#1]{sqrt(#1*cos(x)^2+sin(x)^2)}
\begin{figure}[t]
  \begin{center}
\begin{tikzpicture}
  \begin{polaraxis}[domain=0:360,samples=100,smooth,
    axis line style = {transparent},
    ymax=1, 
    ytick={0.5},
    xticklabels={},
    yticklabels={},
    legend pos=outer north east,
    cycle multi list = {red,blue,green!50!black,black\nextlist solid,densely dashed\nextlist thick}
    ]
    \addplot {\X[0.01]};
    \addlegendentry{$\delta^2=1/100$}
    \addplot {\X[0.0625]};
    \addlegendentry{$\delta^2=1/16$}
    \addplot {\X[0.111111]};
    \addlegendentry{$\delta^2=1/9$}
    \addplot {\X[0.25]};
    \addlegendentry{$\delta^2=1/4$}
    \addplot {\X[0.5]};
    \addlegendentry{$\delta^2=1/2$}
    \addplot {\X[0.75]};
    \addlegendentry{$\delta^2=3/4$}
    \addplot {\X[1]};
    \addlegendentry{$\delta^2=1$ (disk)}
  \end{polaraxis}
\end{tikzpicture}

  \caption{Hippopedes for various choices of $\delta$. For small $\delta$ the curve looks like two circles, while for $\delta=1$ it is a single circle. The hippopede is convex when $\delta^2 \geq 1/2$.   
}
  \label{fig:hippo}
  \end{center}
\end{figure}
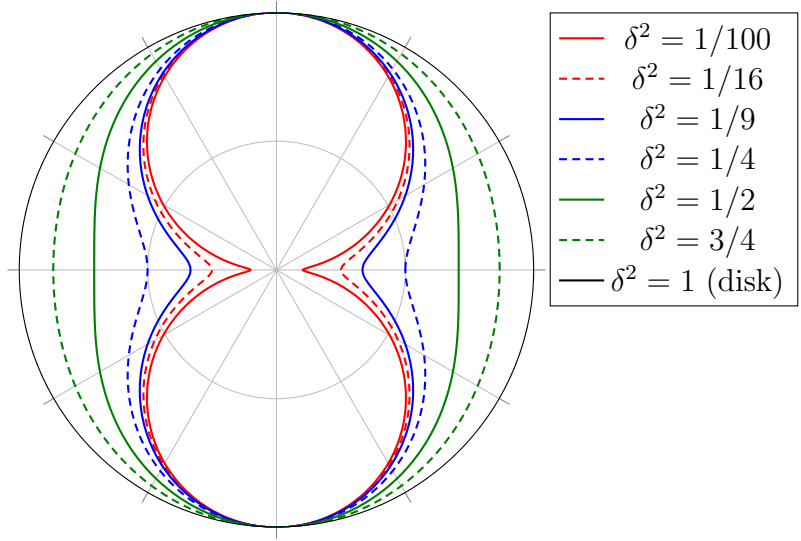

\begin{table}[t]
  \centering
  \begin{tabular}{cccccccc}
    \toprule
  $\delta^2$ & 1/100 & 1/16 & 1/9 & 1/4 & 1/2 & 3/4 & 1 \\
  \midrule
  $\rho_{max}=\rho_2$ & 1.1176 & 1.1016 & 1.0924 & 1.0692 & 1.0281 & 1.0056 & 1 \\  
  $g$      & 2.2751 & 1.4909 & 1.3214 & 1.1378 & 1.0366 & 1.0064 & 1 \\
  $\gamma$ & 2.3733 & 1.6078 & 1.4302 & 1.2112 & 1.0627 & 1.0115 & 1 \\
  \bottomrule
  \end{tabular}
\vspace{6pt}
\caption{Hippopedes centered at the origin, giving values of the ratio $\rho_2$ and the constants $g$ and $\gamma$. The starlike and conformal methods both give better results than summed HPS for hippopedes with $\delta^2 \geq 1/9$, since $g,\gamma<3/2$ in those cases. Both methods are worse than HPS for all $n$ when $\delta$ is small, since then $g,\gamma>2$. See \autoref{sec:example_hippopede}.} 
  \label{tab:hippo}
\end{table}

\subsubsection*{Starlike method}
Let
\[
\delta = \sqrt{1-\e^2} 
\]
where $\e$ is the eccentricity of the ellipse. The hippopede has radius function
\begin{align*}
  R(\theta)=\sqrt{1-(1-\delta^2)\cos^2 \theta} =\sqrt{\sin^2\theta+\delta^2\cos^2\theta},
\end{align*}
as one sees by taking the reciprocal in \eqref{eq:ellipseradius} and then multiplying by $\delta$ (which is a harmless rescaling). Hence
\[
R(\theta)^2+R^\prime(\theta)^2 = \frac{\sin^2\theta+\delta^4\cos^2\theta}{\sin^2 \theta + \delta^2 \cos^2 \theta} \leq 1 .
\]
Note that $R^2+(R^\prime)^2=1$ for $\delta=0$ (two touching disks) and also for $\delta=1$ (one larger disk). Hence $L=2\pi$ in these extreme cases, while in general $L\le 2\pi$.

The first geometric quantity can be evaluated as
\[
  g_0 =\frac{1}{2\pi}\int_0^{2\pi} \frac{R(\theta)^2+R^\prime(\theta)^2}{R(\theta)^2} \, d\theta =\frac{1+\delta^2}{2\delta} ,
\]
which equals the value found for the ellipse in \eqref{eq:ellipseg0}, of course, since $(\log 1/R)^\prime = -(\log R)^\prime$ and the negative sign disappears after squaring. For the second geometric quantity we have
\[
g_1= \frac{\frac{1}{2\pi}\int_0^{2\pi} \big( R(\theta)^2+R^\prime(\theta)^2 \big) \, d\theta}{(L/2\pi)^2} = \frac{1-\delta+\delta^2}{(L/2\pi)^2} .
\]
Note that $g_1=1$ for two touching disks ($\delta=0$) and for a single disk ($\delta=1$). 

\subsubsection*{Conformal method}
The inversion of the hippopede in the unit circle is an ellipse centered at the origin with semiaxes $a=1/\delta$ (evaluate at $\theta=0$) and $b=1$ (evaluate at $\theta=\pi/2$). 
%
%
The Zhukovsky mapping takes the unit disk to the exterior of an ellipse, and the reciprocal of that mapping provides a conformal map onto the hippopede: 
\begin{align*}
  f(z) = \frac{1}{\frac{a-b}{2}z+\frac{a+b}{2}\frac{1}{z}}=\frac{2 \delta z}{{1+\delta+(1-\delta)} z^2} .
\end{align*}
Take the derivative and square to find
\begin{align*}
  p(\theta)^2=|f^\prime(e^{i\theta})|^2=\delta^2\frac{\delta^2 \cos^2\theta+\sin^2\theta}{(\cos^2 \theta + \delta^2 \sin^2 \theta)^2} .
\end{align*}
Since $\int_0^{2\pi} p(\theta)^2 e^{i\theta} \, d\theta = 0$, we have
\begin{align*}
\gamma(\Omega,1)^2 = \gamma_1(p) = \frac{\frac{1}{2\pi}\int_0^{2\pi} p(\theta)^2 d\theta}{(L/2\pi)^2} =\frac{1+\delta^4}{2\delta} \left(\frac{2\pi}{L}\right)^{\! 2} .
\end{align*}
Hence
\begin{align*}
  \gamma(\Omega,1) &=\sqrt{\frac{1+\delta^4}{2\delta}} \, \frac{2\pi}{L} \\
  g&=\sqrt{g_0 g_1}=\sqrt{\frac{1+\delta^2}{2\delta}(1-\delta+\delta^2)} \, \frac{2\pi}{L} \\
\end{align*}
It is easy to check that $g<\gamma$ for all $\delta$.

Note that $g$ and $\gamma$ both blow up as the hippopede approaches two touching disks ($\delta \to 0$), and hence the summed HPS bound is certainly better than ours for small $\delta$. This fact should not be surprising, since the HPS result for the second eigenvalue is optimal for the double-disk.

\subsection{Other computable examples}
\label{sec:example_others}

The alert reader will notice that the starlike method outperforms the conformal one in all three examples so far, namely regular polygons, ellipses, and hippopedes. On the other hand, the conformal method should be preferred over the starlike method in two circumstances:

(i) when the domain is not starlike with respect to any choice of origin (for then the starlike method does not apply) \emph{e.g.}\ the exponential map $f(z)=e^{\pi z}$ takes the unit disk to a domain that wraps around the origin and touches the real axis at $-1$;  
  
(ii) when the domain is starlike but does not possess an explicit radius function $R(\theta)$ (for then the starlike method will be difficult to apply in practice) \emph{e.g.}\ the conformal map $f(z)=z+cz^{N+1}/(N+1)$ takes the disk to an ``$N$-fold lima\c{c}on'' when $0<c<1$, and this domain does not possess an explicit polar representation when $N \geq 3$, so far as we are aware. 

\section*{Acknowledgments}
This work was partially supported by grants from the FRQNT New Researchers Start-up Program (to Alexandre Girouard), Simons Foundation (\#204296 to Richard Laugesen), National Science Foundation grant DMS-0803120, University of Illinois Research Board, and Polish National Science Centre (2012/07/B/ST1/03356 to Bart\-{\l}omiej Siudeja). 

We are grateful to the following research centers for supporting our participation in workshops at which this paper was developed: MFO-Oberwolfach ``Geometric Aspects of Spectral Theory'' (July 2012); de Giorgi Center at the Scuola Normale in Pisa ``New Trends in Shape Optimization'' (July 2012);  Universit\'{e} de Neuch\^{a}tel ``Workshop on Spectral Theory and Geometry'' (June 2013); Banff International Research Station ``Spectral Theory of Laplace and Schr\"{o}dinger Operators" (July 2013).

\newcommand{\doi}[1]{%
 \href{http://dx.doi.org/#1}{doi:#1}}
\newcommand{\arxiv}[1]{%
 \href{http://front.math.ucdavis.edu/#1}{ArXiv:#1}}
\newcommand{\mref}[1]{%
\href{http://www.ams.org/mathscinet-getitem?mr=#1}{#1}}

\end{document}